\documentclass[10pt]{article}
\usepackage{amsmath}
\usepackage{amssymb}
\usepackage{comment}
\usepackage{amsfonts}
\usepackage{amsthm}
\usepackage{rotating}
\usepackage{mathtools}
\usepackage{dsfont}
\usepackage{array}
\usepackage{multirow}
\usepackage{color}
\usepackage{multicol}
\usepackage{stmaryrd}
\usepackage{graphicx}
\usepackage{ulem}
\usepackage{tabularx}
\usepackage{longtable}
\usepackage{lscape}
\usepackage{float}

 \usepackage[ruled,lined,commentsnumbered]{algorithm2e}
\newtheorem{theo}{Theorem}
\newtheorem{proposition}{Proposition}
\newtheorem{lemma}{Lemma}

\newtheorem{rmq}{Remark}
\newtheorem{coro}{Corollary}
\newcommand{\tcr}{\textcolor{black}}

\newcommand{\LW}{\mathcal{L}_{\mathcal{W}}}
\newcommand \cqfd{$\qquad \square$}
\newcommand{\znr}{Z_n^{(r)}}
\newcommand{\znzr}{Z_{n_0}^{(r)}}
\newcommand{\znnr}{Z_{n+1}^{(r)}}
\newcommand{\znrr}{Z_n^{(r+1)}}
\newcommand{\temps}{n}
\numberwithin{ass}{section}
\numberwithin{theo}{section} \numberwithin{proposition}{section}
\numberwithin{coro}{section}
\numberwithin{lemma}{section}
\numberwithin{definition}{section}
 \numberwithin{rmq}{section}

\newcommand{\ER}{\mathbb {R}}\newcommand{\EN}{\mathbb {N}}
\newcommand{\PE}{\mathbb {P}}
\newcommand{\ES}{\mathbb{E}}

\author{S\'ebastien Gadat\footnote{Toulouse School of Economics, UMR 5604, Universit\'e Toulouse I Capitole, France. \ E-mail:\texttt{sebastien.gadat@math.univ-toulouse.fr}.},
Fabien Panloup\footnote{Institut de Math\'ematiques de Toulouse, UMR 5219, 118, route de Narbonne F-31062 Toulouse Cedex 9, France. E-mail: \texttt{fabien.panloup@math.univ-toulouse.fr}},
Sofiane Saadane\footnote{Institut de Math\'ematiques de Toulouse, UMR 5219, 118, route de Narbonne F-31062 Toulouse Cedex 9, France. E-mail: \texttt{sofiane.saadane@math.univ-toulouse.fr} }}
\title{Regret bounds for Narendra-Shapiro bandit algorithms}

\begin{document}
\maketitle







\begin{abstract}
Narendra-Shapiro (NS) algorithms are bandit-type algorithms developed in the 1960s which have been deeply studied in infinite horizon but for which scarce non-asymptotic results exist. In this paper, we focus on a non-asymptotic study of the \textit{regret} and  address the following question: are Narendra-Shapiro bandit algorithms competitive from this point of view? In our main result, we obtain some uniform explicit bounds for the regret of \textit{(over)-penalized}-NS algorithms. \smallskip

\noindent We also extend to the multi-armed case some convergence properties of penalized-NS algorithms towards a stationary Piecewise Deterministic Markov Process (PDMP). Finally, we establish some new sharp mixing bounds for these processes.
\end{abstract}

\bigskip







\section{Introduction}

The so-called Narendra-Shapiro \tcr{bandit} algorithm (referred to as NSa) was introduced in \cite{Norman} and developed in  \cite{Narendra} as a linear learning automata. \tcr{This} algorithm has been primarily considered by the probabilistic community as an interesting benchmark of  stochastic algorithm. More precisely, \tcr{NSa is  an example of  recursive (non-homogeneous) Markovian algorithm, topic} whose almost complete historical overview   may be found in the seminal contributions of \cite{Duflo} and \cite{Kushner}.\smallskip

\noindent  NSa belongs to the large class of  bandit-type policies \tcr{whose principle} may be  sketched as follows:
a $d$-armed bandit algorithm is a procedure designed to determine  which one, among $d$ sources, is the most profitable without spending too much time on the wrong ones. In the simplest case, the sources (or arms) randomly provide some rewards whose values belong to $\{0;1\}$ with Bernoulli laws. The associated probabilities of success $(p_1,...,p_d)$ are unknown to the player and his goal is to determine the most efficient source, \textit{i.e.} the highest probability of success.

\qquad Let us now remind a rigorous definition of admissible sequential policies.
We consider $d$ independent sequences $(A_n^i)_{n \geq 0}$ of \textit{i.i.d.} Bernoulli random variables $\mathcal{B}(p_i)$. Each $A_n^i$ represents the reward associated with the arm $i$ at  time $n$. 
We then consider some sequential predictions where at each stage $n$ a forecaster chooses an arm $I_{\temps}$, receives a reward $A_{\temps}^{I_{\temps}}$ and then uses this information to choose the next arm at step $\temps+1$. 
As introduced in the pioneering work  \cite{Robbins}, the rewards are sampled independently of a fixed product distribution at each step $\temps$. The innovations  here at time $\temps$ are provided by $(I_{\temps},A_{\temps}^{I_{\temps}})$ and we are naturally led to introduce the filtration $(\mathcal{F}_n)_{n \geq 0} := \left(\sigma((I_1,A_{1}^{I_1}), \ldots,(I_{\temps},A_{\temps}^{I_{\temps}}))\right)_{n \geq 0}$. In the following, the sequential admissible policies will be a $(\mathcal{F}_n)_{n \geq 0}$ (inhomogeneous) Markov chain.
We also define another filtration by adding all the events before step $n$ and observe that 
$(\bar{\mathcal{F}}_n)_{n \geq 0} := (\sigma((I_1,(A_1^{j})_{1 \leq j \leq d}) \ldots, (I_n,(A_n^{j})_{1 \leq j \leq d})))_{n \geq 0}.$ To sum-up, 
$\bar{\mathcal{F}}_n$ contains all the results of each arm between time $1$ and $n$ although $\mathcal{F}_n$ only provides partial information about the tested arms.\smallskip


In this paper, we focus on the stochastic NSa whose principle is very simple: it consists in sampling one arm according to a probability distribution on $\{1,\ldots,d\}$, and \tcr{in} modifying this probability distribution  in terms of the reward obtained with the chosen arm. \tcr{From this point of view}, this algorithm bears similarities with the EXP3 algorithm (and many of its variants) introduced in \cite{Auer2}.  Among other close bandit algorithms, one can also  cite the Thompson Sampling strategy where the random selection of the  arm is based on  a Bayesian posterior which is updated after each result. We refer to \cite{Thompson} for a recent theoretical contribution on this algorithm. \smallskip

\noindent Instead of sampling one arm sequentially according to a randomized decision, other algorithms define their policy through a deterministic maximization procedure at each iteration.
 Among them, we can mention the  UCB algorithm \cite{Auer} and its derivatives (including MOSS \cite{AudBub} and KL-UCB \cite{Garivier}), whose dynamics are  dictated by an appropriate empirical upper confidence bound of the estimated best performance.
 \smallskip

\tcr{Let us now present the NSa algorithm. In fact,} we will distinguish two types of NSa: crude-NSa and penalized-NSa. Before going further, let us recall their mechanism in the case of $d=2$ (the general case will be introduced in Section \ref{sec:def}). Designating $X_n$ as the probability of drawing arm 1 at step $n$ and $(\gamma_n)_{n\geq 0}$ as a decreasing sequence of positive numbers that tends to 0 when $n$ goes to infinity, crude-NS is recursively defined by: 
 \begin{equation}\label{eq:bandit}
X_{n+1}=X_n+\begin{cases} \gamma_{n+1} (1-X_n )&\textnormal{if arm $1$ is selected and wins }\\
-\gamma_{n+1} X_n&\textnormal{if another arm is selected and wins }\\
0 &\textnormal{otherwise}\\
\end{cases}
\end{equation}
Note that the construction is certainly symmetric, \textit{i.e.}, $1-X_n$ (which corresponds to the probability of drawing arm 2) has a symmetric dynamics.
The long-time behavior of some NSa was extensively investigated in the last decade. To name a few, in \cite{PLT_trust} and \cite{PL_fast}, some convergence and rate of convergence results are proved. However, these results  strongly depend on both $(\gamma_n)$ and the probabilities of success of the arms. In order to \tcr{get rid of} these constraints, the authors then introduced in 
 \cite{Lamberton_Pages} a  penalized NSa   and proved that this method is an efficient distribution-free procedure, meaning that
 it unconditionally converges to the best arm  on the unknown probabilities $p_1$ and $p_2$. 
 The idea of the penalized-NS algorithm  is to also take  the failures of the player into account  and to reduce  \tcr{the probability of drawing the tested arm} when it loses. Designating  $(\rho_n)_{n\geq 0}$ as a second positive sequence, the dynamics of the penalized NSa  is given by :

\begin{equation}\label{algo:pen_LRI}
X_{n+1}=X_n +\begin{cases} \gamma_{n+1} (1-X_n)&\textnormal{if arm $1$ is selected and wins }\\
-\gamma_{n+1} X_n&\textnormal{if arm $2$ is selected and wins }\\
-\rho_n\gamma_{n+1} X_n&\textnormal{if arm $1$ is selected and loses}\\
\rho_{n+1}\gamma_{n+1} (1-X_n)&\textnormal{if arm $2$ is selected and loses.}
\end{cases}
\end{equation}

\textbf{Performances of  bandit algorithms.} In view of potential applications, it is \tcr{certainly} important to  \tcr{have some informations about} the  performances of the \tcr{used policies}. \tcr{To this end, one first needs to define} what is a ``good" sequencial algorithm.  The primary \tcr{efficiency requirement is the ability of the algorithm} to  \tcr{asymptotically recover} the best arm. In \cite{Lamberton_Pages}, this property is referred to as the \textit{infallibility} of the algorithm. If  
\tcr{without loss of generality, the first arm is assumed to be the best,  ($i.e.$ that $p_1>\max\{p_2,\ldots,p_d\}$) and if $X_n^{(1)}$ denotes the probability of drawing arm $1$, the algorithm is said to be infallible if}
\begin{equation}\label{eq:infall}
\PE(X_n^{(1)}\xrightarrow{n\rightarrow+\infty}1)=1.
\end{equation}

An alternative way for describing the efficiency of a method is to consider the behaviour of the cumulative reward $S_n$ obtained between time $1$ and $n$:
$$
S_n := \sum_{k=1}^{n} A_k^{I_k}.
$$
In particular, in the old paper \cite{Robbins}, Robbins is looking for algorithms such that
$$
p_1 - \frac{\mathbb{E}[S_n]}{n} \xrightarrow{n\rightarrow+\infty}0.
$$
This last property is  weaker than the infallibility of an algorithm since the Lebesgue theorem associated to \eqref{eq:infall} implies the convergence above.\\

 A much stronger requirement involves the \textit{regret} of the algorithm.
The regret measures the gap between the cumulative reward of the best player and the one induced by the policy. The regret $R_n$ is the $\bar{\mathcal{F}}_n$-measurable random variable defined as:
\begin{equation}\label{eq:defregret}
R_n := \max_{1 \leq j \leq d} \sum_{k=1}^n \left[A_{k}^j - A_{k}^{I_{k}}\right].
\end{equation}
A good strategy corresponds to a selection procedure that minimizes the expected regret $ \mathbb{E} R_n$, optimal ones being referred to as \textit{minimax} strategies. 

The former expected regret cannot be easily handled and is generally replaced in statistical analysis by the \textit{pseudo-regret} defined as
\begin{equation}\label{eq:defpseudo}
\bar{R}_n := \max_{1 \leq j \leq d} \mathbb{E} \sum_{k=1}^n  \left[A_{k}^j - A_{k}^{I_{k}}\right].
\end{equation}
Since $p_1 > p_j, \forall j \neq 1$, $\bar{R}_n$ can also be written as
\begin{eqnarray*}
\bar{R}_n &=&  \sum_{{k}=1}^n\mathbb{E}(A^1 _k)-\mathbb{E} \left( \sum_{{k}=1}^n A_{{k}}^{I_{{k}}}\right) = n \left( p_1 - \frac{\mathbb{E}[S_n]}{n}\right).
\end{eqnarray*}

A low \textit{pseudo-regret} property then means that the quantity
$$
n \left( p_1 - \frac{\mathbb{E}[S_n]}{n}\right)
$$
has to be small, in particular sub-linear with $n$. The  quantities $R_n$ and $\bar{R}_n$ are closely related and it is reasonable to study the pseudo-regret instead of the true regret, owing to the next proposition:
\begin{proposition}\label{prop:regret_vs_pseudo}
$(i)$ For any $(\mathcal{F}_n)_{n \geq 0}$-measurable strategy, we obtain after $n$ plays:

$$
0 \leq \mathbb{E} R_n - \bar{R}_n \leq  \sqrt{\frac{n\log d}{2}}.
$$
$(ii)$ Furthermore,  for every integer $n$ and $d$ and for any (admissible) strategy,
$$\sup_{p_1 > p_2 \geq \ldots \geq p_d} \ES[{R}_n]\ge \frac{1}{20}\sqrt{nd}.$$
\end{proposition}
We refer to Proposition 34 of \cite{Audibert} for a detailed proof of $(i)$ and to Theorem 5.1 of \cite{Auer2} for $(ii)$. As mentioned in $(ii)$, the bounds are distribution-free (uniform in $p$).\footnote{The rate orders are strongly different if a dependence in $p$ is allowed.} Since the MOSS method of \cite{AudBub} satisfies $\bar{R}_n \leq 25 \sqrt{n d}$, $(i)$ and $(ii)$ show that a non-asymptotic distribution-free minimax  rate is on the order of $\sqrt{n}$.
\smallskip

In particular, a fallible algorithm (meaning that $\PE(X_n\xrightarrow{n\rightarrow+\infty}1)<1$) necessarily generates a linear regret and is not optimal. For example, in the case $d=2$, the dependence of $\bar{R}_n$ in terms of $(X_n)$ is as follows:
\begin{eqnarray}\label{eq:regret_sequentiel}
\bar{R}_n&=& p_1 n-\sum_{k=1}^n (p_1 \ES[X_k]+p_2\ES[1-X_k])=(p_1-p_2)\ES\left[\sum_{k=1}^n (1-X_k)\right] \\
& \gtrsim & (p_1-p_2) \mathbb{P}(X_{\infty}=0) \times n. \nonumber
\end{eqnarray}

\textbf{Objectives.} In this paper, we therefore propose to focus on the regret and to answer to the question ``Are NSa competitive from a regret viewpoint? In the case of positive answer, what are the associated upper-bounds ?"\smallskip

\noindent \tcr{Due to some too restrictive conditions of infallibility, it will be seen that} the crude-NSa cannot be competitive from a regret point of view. As mentioned before, the penalized NSa  is more robust and is \textit{a priori} more appropriate for this problem. More precisely, the penalty induces more balance between exploration and exploitation, $i.e.$ between playing the best arm (the one in terms of the past actions) and exploring new options (playing the suboptimal arms). In this paper, we are going to prove that, up to a slight reinforcement, it is possible to obtain some competitive bounds for the regret of this procedure. The slightly modified penalized algorithm will be referred to as the \textit{over-penalized}-algorithm below.

\textbf{Outline.} The paper is organized as follows : Section \ref{sec:def22} provides some basic information about the crude NSa.   
 Then, in Section \ref{sec:algo}, \tcr{after some background on the penalized Nsa, we introduce a new algorithm called over-penalized NSa}. 

\noindent  Section \ref{sec:mainresults} is devoted to the main results: in Theorem \ref{theo:2bras2}, we establish an  upper-bound of the \textit{pseudo-regret} $\bar{R}_n$ for the over-penalized algorithm in the two-armed case 
and also show a weaker result for the penalized NSa.

\noindent In this section, we also extend to the multi-armed case some existing convergence and rate of convergence results of the two-armed algorithm. In the ``critical'' case (see below for details), the normalized algorithm converges in distribution toward a PDMP (Piecewise Deterministic Markov Process). We develop a careful study of its ergodicity and bounds on the rate of convergence to equilibrium are established. It uses a non-trivial coupling strategy to derive explicit rates of convergence in Wasserstein and total variation distance. The dependence of these rates are made explicit with the several parameters of the initial Bandit problem.

\noindent The rest of the paper is devoted to the proofs of the main results: Section \ref{sec:proof_regret} \tcr{is dedicated to  the regret analysis}, and Section \ref{sec:appendix_multi} establishes the weak limit of the rescaled multi-armed bandit algorithm. Finally, Section \ref{sec:ergodicity} includes all the proofs of the ergodic rates.

\section{Definitions of the NS algorithms}\label{sec:def}

\subsection{Crude NSa and regret}\label{sec:def22}
The crude NSa  \eqref{eq:bandit} is rather simple: it defines a $({\cal F}_n)_{n \geq 0}$ Markov chain $(X_n)_{n \geq 0}$ and  $I_{\temps}$ is a random variable satisfying: 

$$\mathbb{P}(I_{\temps+1} =1\vert \mathcal{F}_n) = X_n \quad \textnormal{and}\quad \mathbb{P}(I_{\temps+1} =2\vert \mathcal{F}_n) = 1-X_n$$ The arm $I_{\temps+1}$ is selected at step $\temps+1$ with the current distribution $(X_{\temps},1-X_{\temps})$ and is evaluated. In the event of success, the weight of the arm $I_{\temps+1}$ is increased  and the weight of the other arm is decreased by the same quantity. The algorithm can be rewritten in a more  \tcr{concise} form as:
\begin{eqnarray}\label{algo:LRI}
X_{n+1} = X_n  + \gamma_{\temps+1}(\mathds{1}_{I_{\temps+1}=1}-X_n )A_{n+1}^{I_{n+1}}.
\end{eqnarray}

The arm $i$ at step $n$ succeeds with the probability $p_i= \mathbb{P}(A^i_n=1)$ and we suppose $w.l.o.g.$ that $p_1 > p_2$ so that the arm 1 is the optimal one.

As pointed \tcr{in} \eqref{eq:regret_sequentiel}, we obtain that
$$
\bar{R}_n=(p_1-p_2)\ES\left[\sum_{k=1}^n (1-X_k)\right].$$

This formula is important regarding the fallibility of an algorithm. In particular, it is shown in \cite{PL_fast} that for any choice $\gamma_n = C(n+1)^{-\alpha}$ with $\alpha \in (0,1)$ and $C>0$ or $\gamma_n = C/(n+1)$ with $C>1$, the NSa  \eqref{algo:LRI} may be fallible: some parameters $(p_1,p_2)$ exist such that $(X_n)_{n \geq 0}$ a.s. converges to a binary random variable $X_{\infty}$ with $\mathbb{P}(X_{\infty}=0)>0$. In this situation, for  large enough $n$, we have:
$$
\bar{R}_n \gtrsim (p_1-p_2) \mathbb{P}(X_{\infty}=0) \times n >> \sqrt{n}
$$

It can easily be concluded that this method cannot induce a competitive policy since some ``bad" values of the probabilities $(p_1,p_2)$ generate a linear regret.

\noindent 
%




\subsection{Penalized and over-penalized two-armed NSa}\label{sec:algo}
\paragraph{Penalized NSa.}
A major difference between the crude NSa and its penalized counterpart introduced in \cite{Lamberton_Pages} relies on the exploitation of the failure of the selected arms. The
crude NSa \eqref{eq:bandit} only uses the sequence of successes to update the probability distribution $(X_{\temps},1-X_\temps)$ since the value of $X_{\temps}$ is modified \textit{iff} $A_{\temps}^{I_{\temps}}=1$. In contrast, the penalized NSa  \eqref{algo:pen_LRI} also uses the information generated by a potential failure of the arm $I_{\temps+1}$.
More precisely, in the event of success of the selected arm $I_{\temps+1}$, this penalized NSa mimics the crude NSa, whereas in the case of failure, the weight of the selected arm is now multiplied (and thus decreased) by a factor $(1-\gamma_{\temps+1} \rho_{\temps+1})$ (whereas the probability of drawing the other arm is increased by the corresponding quantity). 
For the penalized NSa, the update formula of $(X_{\temps})_{\temps \geq 1}$ can be written in the following way:
\begin{eqnarray}
X_{\temps+1}& =& X_{\temps} + \gamma_{\temps+1} \left[ \mathds{1}_{I_{\temps+1}=1} - X_{\temps} \right] A_{\temps+1}^{I_{\temps+1}} \nonumber \\ &&- \gamma_{\temps+1} \rho_{\temps+1} \left[ X_{\temps}  \mathds{1}_{I_{\temps+1}=1}  - (1-X_{\temps}) \mathds{1}_{I_{\temps+1}=2} \right] (1-A_{\temps+1}^{I_{\temps+1}}).\label{eq:pen_LRId2}
\end{eqnarray}
\paragraph{Over-penalized NSa.}
\tcr{In view of the minimization of the regret}, we will show that it may be  useful to reinforce the penalization.
For this purpose, we introduce a slightly ``over-penalized''  NSa where a player is also (slightly) penalized if it wins:
 \begin{itemize}
\item If player 1 wins, then with probability $1-\sigma$ it is penalized by  a factor $\gamma_{n+1}\rho_{n+1}X_n$.
\item If player 2 wins, then with probability $1-\sigma$ arm 1 is increased by a factor of $\gamma_{n+1}\rho_{n+1}(1-X_n)$.
\end{itemize}  
The over-penalized-NSa can be written as follows
\begin{eqnarray}
\lefteqn{X_{\temps+1}^\sigma = 
 X_{\temps}^\sigma + \gamma_{\temps+1} \left[ \mathds{1}_{I_{\temps+1}=1} - X_{\temps}^\sigma \right] A_{\temps+1}^{I_{\temps+1}}}\nonumber\\
  &-&  \gamma_{\temps+1} \rho_{\temps+1} \left[ X_{\temps}^\sigma  \mathds{1}_{I_{\temps+1}=1}  - (1-X_{\temps}^\sigma) \mathds{1}_{I_{\temps+1}=2} \right] \left(
1-A_{\temps+1}^{I_{\temps+1}}B_{\temps+1}^\sigma\right) \label{eq:pen_LRId3}
\end{eqnarray}
where $(B_\temps^\sigma)_{\temps}$ is a sequence of i.i.d. r.v. with a Bernoulli distribution ${\cal B}(\sigma)$, meaning that $\mathbb{P}(B_\temps^\sigma=0)=1-\sigma$. Moreover, these r.v. are  independent of $(A_\temps^j)_{n,j}$ and in such a way that for all $\temps\in\mathbb{N}$, $B_\temps^\sigma$ and $I_\temps$ are also independent. It should be noted that 
$$1-A_{\temps}^{I_{\temps}}B_\temps^\sigma=\left[1-A_\temps^{I_{\temps}}\right]+A_\temps^{I_{\temps}}(1-B_\temps^\sigma).$$
\tcr{In fact, this  slight over-penalization  of the successful arm (with  probability $\sigma$) can be viewed as an additional statistical excitation which helps the stochastic algorithm to escape from local traps}. The case $\sigma=1$ corresponds to the penalized NSa  \eqref{eq:pen_LRId2}, whereas when $\sigma=0$, the arm is always penalized when it plays. \tcr{In particular, this modification implies}  that the increment of $X_\temps^\sigma$ is slightly weaker than in the previous case when the selected arm wins.\smallskip
\paragraph{Asymptotic convergence of the penalized NSa.}
\noindent Before stating the main results, we need to understand which regret $\bar{R}_n$ could be reached by  penalized and over-penalized NSa. We recall (in a slightly less general form) the convergence results of Proposition 3, Theorems 3 and 4 of \cite{Lamberton_Pages}. 
\noindent
\begin{theo}[Lamberton \& Pages, \cite{Lamberton_Pages}]\label{theoLP}
Let  $0 \leq p_2 < p_1 \leq 1$ and $\gamma_{\temps} = \gamma_1 \temps^{- \alpha}$ and $\rho_{\temps} =\rho_1  \temps^{-\beta}$  with $(\alpha,\beta) \in (0,+\infty)$ and $(\gamma_1,\rho_1)\in(0,1)^2$. Let $(X_\temps)_{\temps}$ be the algorithm given by \eqref{eq:pen_LRId2}.
 
\begin{itemize}
\item[$i)$] If $0<\beta \leq \alpha$ and $\alpha+\beta\leq 1$,  the penalized two-armed bandit is {\em infallible}.
\item[$ii)$] Furthermore, if $0<\beta < \alpha$ and $\alpha+\beta < 1$, then
$\displaystyle
\frac{1-X_{\temps}}{\rho_{\temps}}  \longrightarrow \frac{1-p_1}{p_1-p_2} \quad \text{a.s.}
$
\item[$iii)$] If $\alpha = \beta \leq 1/2$ and $g=\gamma_1/\rho_1$:
$\displaystyle
\frac{1-X_{\temps}}{\rho_{\temps}} \overset{w^*}\longrightarrow \mu,
$
where $\overset{w^*}\longrightarrow$ stands for the convergence in distribution and $\mu$ is the stationary distribution  of the PDMP whose generator $\mathcal{L}$ acts on $\mathcal{C}_c^1(\mathbb{R}_+)$  as
$$
\forall f \in \mathcal{C}_c^1(\mathbb{R}_+) \qquad 
\mathcal{L}f(y) = p_2 y \frac{f(y+g)-f(y)}{g} + (1-p_1-p_1 y)f'(y).
$$
\end{itemize}
\end{theo}
In view of Theorem \ref{theoLP}, we can use formula \eqref{eq:regret_sequentiel} to obtain

\begin{eqnarray}\label{eq:regret_2bras}
\bar{R}_n& =& (p_1-p_2)    \sum_{{k}=1}^n \rho_{{k}} \mathbb{E}  \left(\frac{1-X_{{k}}}{\rho_{{k}}}\right).
\end{eqnarray}
We then obtain the key observation
\begin{equation}\label{supnnnn}
\sup_{n\in\mathbb{N}}\ES\left[\frac{1-X_n}{\rho_n}\right]\leq C <+\infty\;\Longrightarrow\; \bar{R}_n \le C(p_1-p_2)\sum_{k=1}^n\rho_k, 
\end{equation}
where $C$ is a constant that may depend on $p_1$ and $p_2$.  According to \tcr{Theorem} \ref{theoLP}, it seems that the potential optimal choice corresponds to the one of $(iii)$. Indeed, the infallibility occurs only when $\alpha \geq \beta$ and $\alpha+\beta \leq 1$ and Equation \eqref{eq:regret_2bras} suggests that $\beta$ should be chosen as large as possible to minimize the r.h.s. of \eqref{supnnnn}, leading to $\alpha=\beta=1/2$.  This  is why in the following, we will focus on the case:
\begin{equation}\label{eq:defpaspoids}
\gamma_n=\frac{\gamma_1}{\sqrt{n}}\quad\textnormal{and}\quad \rho_n=\frac{\rho_1}{\sqrt{n}}.
\end{equation}

\subsection{Over-penalized multi-armed NSa}
We generalize the definition of the penalized and over-penalized NSa to the $d$-armed case, with $d \geq 2$.
\tcr{Let} $p=(p_1,\ldots,p_d) \in (0,1)^d$ and \tcr{assume that $A_n^j \sim \mathcal{B}(p_j)$ ($p_i$ the probability of success of  arm $i$)}. The over-penalized NSa recursively {defines}  a sequence of probability measures \tcr{on $\{1,\ldots,d\}$ denoted by $(\Pi_n)_{n \geq 1}$} where $\Pi_n = (X_n ^1,...,X_n ^d)$. At step $n$, the arm $I_{n+1}$ is sampled according to the discrete distribution $X_n$ and\ tcr{then} tested through the computation of $A_{n+1}^{I_{n+1}}$.
 Setting $j\in \lbrace 1,...,d\rbrace$, the multi-armed NSa is defined by: 

\begin{eqnarray}
\lefteqn{X_{\temps+1}^j= 
  X_{\temps}^j + \gamma_{\temps+1} \left[ \mathds{1}_{I_{\temps+1}=j} - X_{\temps}^j \right]  A_{\temps+1}^{I_{\temps+1}}}\nonumber\\
  &-&   \gamma_{\temps+1} \rho_{\temps+1} X_{\temps}^{I_{\temps+1}}  (1-A_{\temps+1}^{I_{\temps+1}}B_{\temps+1}^\sigma) \left[\mathds{1}_{I_{\temps+1}=j} - \frac{1-\mathds{1}_{I_{\temps+1}=j}}{d-1} \right].\label{eq:pen_bandit}
\end{eqnarray}



 In contrast with the two-armed case, we have  to choose how to distribute the penalty to the other arms when $d>2$. The (natural) choice in \eqref{eq:pen_bandit} is to divide it fairly, $i.e.$, to spread it uniformly over the other arms. Note that alternative algorithms (not studied here) could be considered. 
\section{Main Results}\label{sec:mainresults}

\subsection{Regret of the over-penalized two-armed bandit}
First, we provide some uniform upper-bounds for the two-armed $\sigma$-over-penalized NSa . Our main result is Theorem \ref{theo:2bras2}.
Before stating it, we choose to state a new result when $\sigma=1$, $i.e.$ for the ``original'' penalized NSa  introduced in \cite{Lamberton_Pages}.
\begin{theo}\label{theo:2bras1}
Let $(X_n)_{n\ge0}$ be the two-armed penalized NSa defined by \eqref{eq:pen_LRId2} with $(\gamma_n,\rho_n)_{n\ge1}$ defined by \eqref{eq:defpaspoids} with $(\gamma_1,\rho_1)\in(0,1)^2$. Then, for every $\delta\in(0,1)$, a positive $C_\delta$ exists such that:
$$\forall n\in\EN,\quad \sup_{{(p_1,p_2)\in[0,1]},p_2\le p_1\wedge (1-\delta)}\bar{R}_n\le C_\delta\sqrt{n}.$$
\end{theo}
\begin{rmq} The upper bound of the original penalized-NS algorithm is not completely uniform. From a theoretical point of view, there is not enough  penalty when $p_2$ is too large, which in turn generates a deficiency of the mean-reverting effect for the sequence $((1-X_n)/\rho_n)_{n\ge1}$ when $X_n$ is close to $0$. 
In other words, the trap of the stochastic algorithm near $0$ is not enough repulsive and Figure \ref{fig1:regret} below shows that this
problem also appears numerically and suggests a logarithmic explosion of $\displaystyle{\sup_{p_1<p_2} {\bar{R}_n}/{\sqrt{n}}}$.  \end{rmq}

This explains the interest of the over-penalization, illustrated by the next result, \tcr{which is the main theorem of the paper}.
\begin{theo}\label{theo:2bras2}
Let $(X_n)_{n\ge0}$ be the two-armed $\sigma$-over-penalized NSa defined by \eqref{eq:pen_LRId3} with $\sigma\in[0,1)$ and  \tcr{$(\gamma_n,\rho_n)_{n\ge1}$ defined by \eqref{eq:defpaspoids}} with $(\gamma_1,\rho_1)\in(0,1)^2$. Then,\\
 
\noindent (a) A $C_\sigma(\gamma_1,\rho_1)$ exists such that:
$$\forall n\in\EN,\quad \sup_{{(p_1,p_2)\in[0,1]},p_2 < p_1}\bar{R}_n\le C_\sigma(\gamma_1,\rho_1) \sqrt{n}.$$
(b) Furthermore, the choice $\sigma=0$, $\gamma_n  = 2.63 \rho_n = 0.89/ \sqrt{n}$ yields
 \begin{equation}\label{eq:35bound}
 \forall n\in\EN,\quad \sup_{{(p_1,p_2)\in[0,1]},p_2 < p_1}\bar{R}_n\le 31.1 \sqrt{2n}.
 \end{equation}
\end{theo}
\begin{rmq} At the price of technicalities, $C_\sigma$ could be made explicit in terms of  $\gamma_1$ and $\rho_1$  for every $\sigma>0$. The second bound is obtained by an optimization of \tcr{$C_0(\gamma_1,\rho_1)$ (see \eqref{eq:optimC} and below)}. 
\end{rmq}

\begin{figure}[h]
\begin{center}
\includegraphics[scale=0.35]{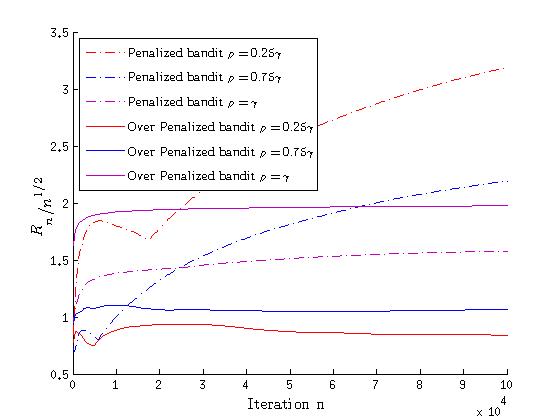}
\includegraphics[scale=0.45]{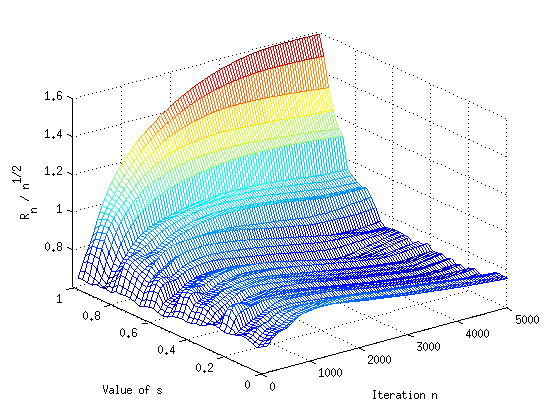} 
\caption{\label{fig1:regret} \tcr{Evolution of $n\mapsto \sup_{{(p_1,p_2)\in[0,1]},p_2\le p_1}\frac{\bar{R}_n}{\sqrt{n}}$ for the over-penalized algorithm (with $\sigma=0$) and comparison with EXP3 and KL-UCB}.}
\end{center}
\end{figure}

Figure \ref{fig1:regret} presents on the left side a numerical approximation of $\displaystyle n\mapsto \sup_{p_2< p_1}R_n/\sqrt{n}$ for the penalized and over-penalized algorithms. The continuous curves indicate that the upper bound \textit{$31.1\sqrt{2}$} in Theorem \ref{theo:2bras2} is not sharp since the over-penalized NSa satisfies a uniform upper-bound on the order of $0.9  \sqrt{n}$. 
This bound is obtained with a small $\sigma$ (as pointed in Theorem \ref{theo:2bras2}), and $\gamma_n=\frac{1}{\sqrt{4+n}}=4 \rho_n$ (red line in Figure \ref{fig1:regret} (left)), suggesting that the rewards should  \textit{always} be over-penalized with $\rho_n=\frac{\gamma_n}{4 }$.

The right-hand side of Figure \ref{fig1:regret} focuses on the behavior of the regret with $\sigma$. The map $(n,\sigma)\mapsto \displaystyle{\sup_{p_1<p_2} {\bar{R}_n}/{\sqrt{n}}}$ confirms the influence of the over-penalization and indicates that to obtain optimal performances for the cumulative regret, we should use a low value of $\sigma$ between $0$ and $3/5$. The importance of this choice of $\sigma$ seems relative since the behaviour of the over-penalized bandit is stable on this interval.  
The best numerical choice is attained for $\sigma=1/4$ and $\rho_n = \frac{1}{4} \gamma_n$ and permits to achieve a long-time behavior of $\bar{R}_n/\sqrt{n}$ of the order $3/4$ (see Figure \ref{fig2:regret}, red line).

\begin{figure}[h!]
\begin{center}
\includegraphics[scale=0.5]{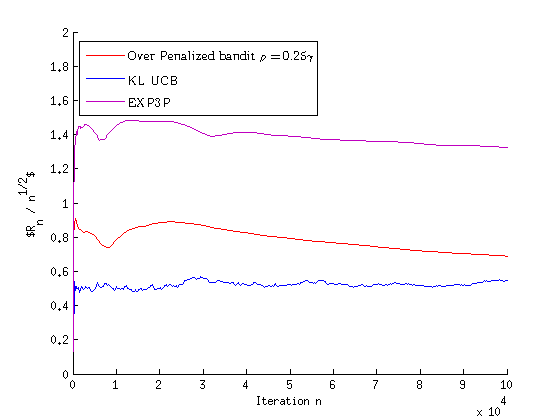} 
\caption{\label{fig2:regret} \tcr{Evolution of $n\mapsto \sup_{{(p_1,p_2)\in[0,1]},p_2\le p_1}\frac{\bar{R}_n}{\sqrt{n}}$ for the over-penalized algorithm (with $\sigma=\dfrac 1 4$) and comparison with EXP3 and KL-UCB}.}
\end{center}
\end{figure}

Finally, the statistical performances of the over-penalized NSa are compared with some classical bandit algorithms: KL-UCB algorithm (see \textit{e.g.} \cite{Garivier} and the references therein) and EXP3 (see \cite{Auer2}). These two algorithms are anytime policies that are known to be minimax optimal with a cumulative minimax regret of the order $\sqrt{n}$. Figure \ref{fig2:regret} shows that the performances of the over-penalized NSa are located between the  one of the KL-UCB
 algorithm  and of the EXP3 algorithm (our simulations suggest that the uniform bounds of KL-UCB and EXP3 are respectively $1/2$ and $3/2$). 
Also, it is worth noting that the simulation cost of the over-penalized NSa is strongly weaker than the initial UCB algorithm  (the phenomenon is increased when compared to KL-UCB, which requires an additional difficulty for the computation of the upper confidence bound at each step): the same amount of Monte-Carlo simulations for the over-penalized NSa is almost hundred times faster than the KL-UCB runs in equivalent numerical conditions.

\subsection{Convergence of the  multi-armed over-penalized bandit}
We first extend  Theorem \ref{theoLP} of \cite{Lamberton_Pages} to  the over-penalized NSa in the multi-armed situation. The result describes the pointwise convergence.
\begin{proposition}[Convergence of the multi-armed over-penalized bandit]\label{prop:multi1}
Consider $p_d \leq \ldots \leq p_2<p_{1}$  and  $\gamma_{\temps} = \gamma_1 \temps^{- \alpha}, \rho_{\temps} =\rho_1  \temps^{-\beta}$  with $(\alpha,\beta) \in (0,+\infty)$ and $(\gamma_1,\rho_1)\in(0,1)^2$. Algorithm \eqref{eq:pen_LRId2} with $\sigma\in(0,1]$ satisfies
 
\begin{itemize}
\item[$i)$] If $0<\beta \leq \alpha$ and $\alpha+\beta\leq 1$, then \tcr{$\lim_{n\rightarrow+\infty} \Pi_n=(1,0,\ldots,0)$} $a.s$.
\item[$ii)$] Furthermore, if $0<\beta < \alpha$ and $\alpha+\beta < 1$, then:
$$
\forall i\in\{2,\ldots,d\},\quad\frac{X_n^i}{\rho_{\temps}}  \longrightarrow \frac{1-\sigma p_{1}}{(d-1)(p_{1}-p_{i})} \qquad \text{a.s.}
$$
\end{itemize}
\end{proposition}
Proposition \ref{prop:multiweak} provides a description of the behavior of the \textit{normalized} NSa while considering $Y_{n,j} = \frac{X_{n,j}}{\rho_n}$. It states that  $(Y_{n,.})_{n \geq 0}$ converges to the dynamics of a \textit{Piecewise Deterministic Markov Process} (referred to as PDMP below).

\begin{proposition}[Weak convergence of the over-penalized NSa]\label{prop:multiweak}
Under the assumptions of Proposition \ref{prop:multi1}, if $\alpha = \beta \leq 1/2$ and $g=\gamma_1/\rho_1$, then:
$$
\frac{1}{\rho_{\temps}}\left(X_{n,2},\ldots,X_{n,d}\right) \overset{w^*}\longrightarrow \mu_{d},
$$
where $\mu_d$ is the (unique) stationary distribution of the Markov process whose  generator $\mathcal{L}_d$ acts on compactly supported functions $f$ of $\mathcal{C}^1((\mathbb{R}_+)^{d-1})$  as follows:

\begin{eqnarray}\label{generateur_dbras}
{\cal L}_d f(y_{2},...,y_{d})&=&\sum\limits_{i=2,...,d}\frac{p_{i}y_{i}}{g}(f(y_{2},...,y_{i}+g,...,y_{d})-f(y_{2},...,y_{i},...y_{d}))\nonumber\\
&+&\sum\limits_{i=2,...,d}(\frac{1-\sigma p_{1}}{d-1}-p_{1}y_{i})\partial_{i}f(y_{2},...,y_{d}).
\end{eqnarray}
\end{proposition}

\subsection{Ergodicity of the limiting process}
In this section, we focus on the long time behavior of the limiting Markov process that appears (after normalization) in Proposition \ref{prop:multiweak}. As mentioned before, this process is a PDMP and its long time behavior can be carefully studied with some arguments in the spirit of \cite{Malrieu}. We also learned about the existence of a close study in the PhD thesis of Florian Bouguet (some details may be found in \cite{bouguet}). Such properties are stated for both the one-dimensional and the multidimensional cases. 
\subsubsection{One-dimensional case} 
Setting
$$a = 1-p_{1}, \;b=p_{1},\; g=\frac{\gamma_1}{\rho_1}, c=\frac{p_{2}}{g},$$
 the generator $\mathcal{L}$ given by Proposition \ref{prop:multiweak} may be written as: 
\begin{eqnarray}\label{generateur}
\forall f \in \mathcal{C}^1(\ER_+^*,\ER) \quad 
\mathcal{L} f(x) = \underbrace{(a-bx)f'(x)}_{\text{deterministic } \text{part}}+\underbrace{cx}_{\text{jump }\text{rate}}\underbrace{(f(x+g)-f(x))}_{\text{jump }\text{size}}.
\end{eqnarray}

In what follows, we will assume that $a,$ $b$, $c$ and $g$ are  positive numbers. We can see in $\mathcal{L}$ two parts.
On the one hand, the deterministic flow that guides the PDMP between the jumps is given by:
$$\left\{
\begin{array}{ll}
\partial_{t}\phi(x,t) &= (a-bx)\partial_{x}\phi(x,t)\\
\phi(x,0) &= x \in\ER_+^*
\end{array}
\right.$$
so that 
$$\phi(x,t) = \frac{a}{b} + \left(x-\frac{a}{b}\right)e^{-bt}. $$
Hence,  if $x>\frac{a}{b}$ (resp. $x<\frac{a}{b}$), $t\mapsto \phi(x,t)$ decreases (resp. increases) and converges exponentially fast to $\frac{a}{b}$.\smallskip

On the other hand,  the PDMP possesses some positive jumps that occur  with a Poisson intensity ``$c.x$'', whose size is deterministic and equals to $g$. 

From the finiteness and positivity of $g$, it is easy to show that for every positive starting point, the process is $a.s.$ well-defined on $\ER_+$, positive and does not explode in finite time. The fact that the size of the jumps is deterministic is less important and what follows could  easily be generalized to a random size $g$ (under adapted integrability assumptions). 
 In Figure \ref{fig:traj} below, some paths of the process are represented with different values of the parameters.
\begin{figure}[h]
\begin{center}
\includegraphics[scale=0.2]{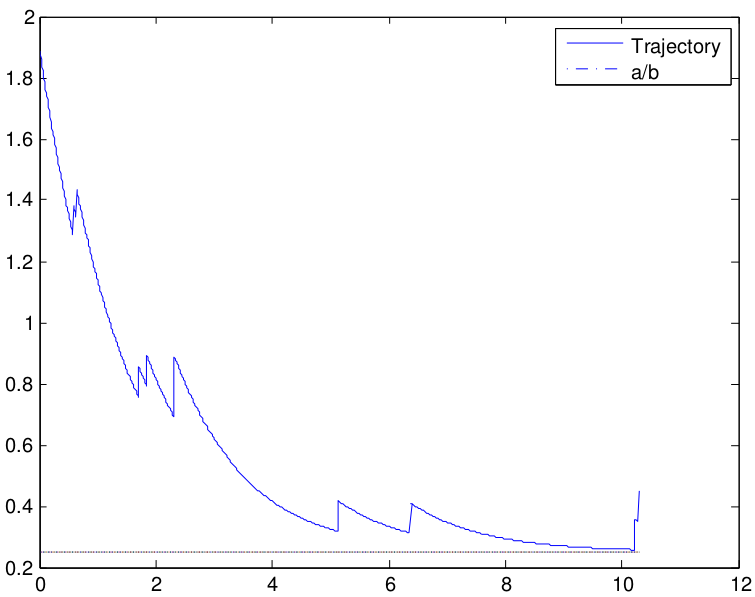}\includegraphics[scale=0.24]{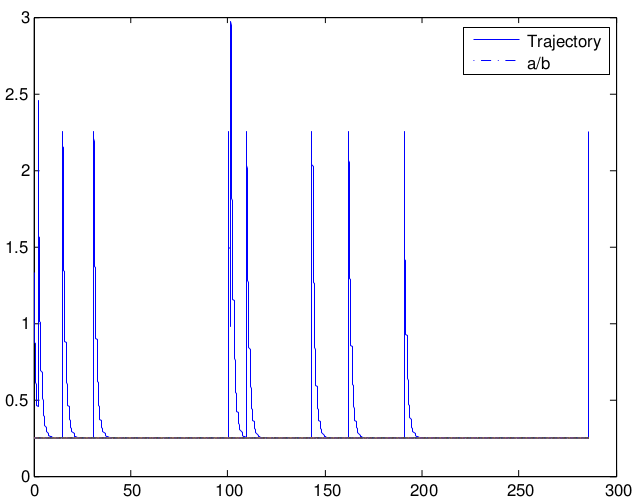} \\
\includegraphics[scale=0.2]{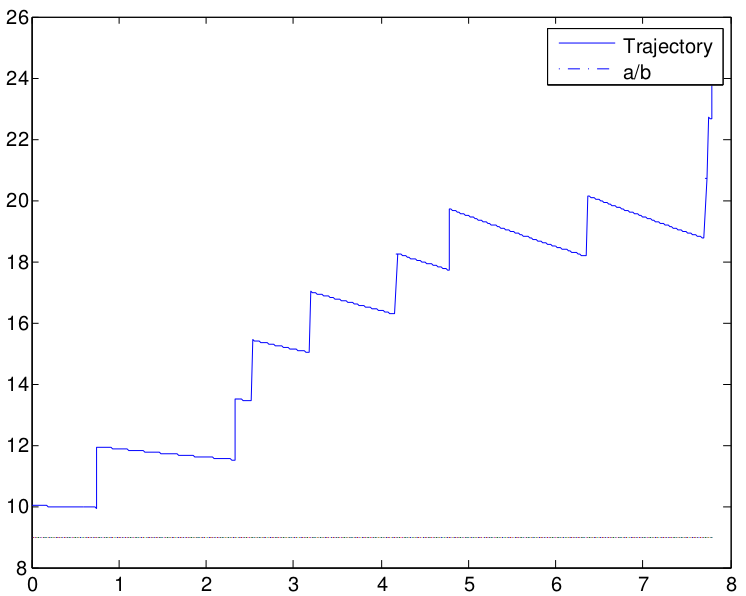}\includegraphics[scale=0.2]{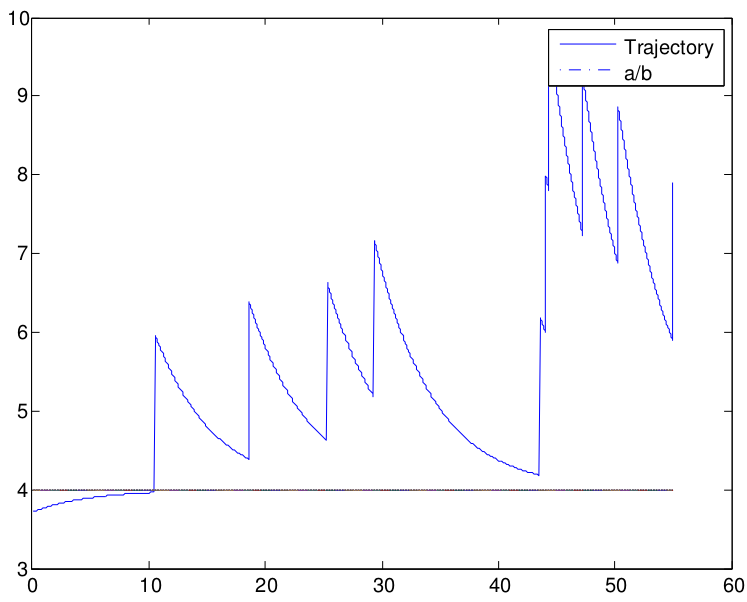} 
\caption{\label{fig:traj} Exact simulation of trajectories of a process driven by (\ref{generateur}) when $g=0.1,a=0.2,b=0.8,c=0.2$ (top left) $g=2,a=0.2,b=0.8,c=0.1$ (top right), $g=2, a=0.9,b=0.9,c=0.15$ (bottom left) and $g=2,a=0.8,b=0.2,c=0.05$ (bottom right).}
\end{center}
\end{figure}

\subsubsection{Convergence results}
As pointed out in Figure \ref{fig:traj}, the long-time behavior of the process certainly depends on the relationship between the mean-reverting effect generated by ``$-bx$'' and the frequency and size of the jumps.
\paragraph{Invariant measure}
The process \eqref{generateur} possesses a unique invariant distribution if $b-cg>0$. Actually, the existence is ensured by the fact that $V(x)=x$ is a Lyapunov function for the process since
\begin{equation*}
\forall x\in\ER_+^*,\quad  {\cal L} V(x)= a-(b-cg)x=a-(b-cg) V(x)
\end{equation*}
Among other arguments, the uniqueness is ensured by Theorem \ref{theo:pdmp} (the convergence in Wasserstein distance of the process toward the invariant distribution
implies in particular its uniqueness). We denote it by $\mu_\infty$ below. It could  also be shown that ${\rm Supp}(\mu_\infty)=(a/b,+\infty)$, that the process is strongly ergodic on $(a/b, + \infty)$ (see \cite{EthierKurtz} for some background) and that if $b-cg>0$, the process explodes when $t\rightarrow+\infty$ (this case corresponds to the bottom left-hand side of Figure \ref{fig:traj}). Finally, it should be noted that for the limiting PDMP of the bandit algorithm,
$$b-cg=p_1-p_2=\pi$$
and thus, the ergodicity condition coincides with the positivity of $\pi$.
\paragraph{Wasserstein results}
We aim to derive rates of convergence for the PDMP toward $\mu_{\infty}$ for two distances, namely the Wasserstein distance and the total variation distance. Rather different ways to obtain such results exist using coupling arguments or PDEs. We  use coupling techniques here that are consistent with the work of \cite{BCG+} and \cite{Cloez}. Before stating our results, let us recall that the $p$-Wasserstein distance is defined for any probability measures $\mu$ and $\nu$ on $\mathbb{R}^d$ by:
\begin{eqnarray*}
\mathcal{W}_{p}(\mu,\nu) = \inf\left\{\mathbb{E}\left(|X-Y|^p)\right)^{\frac{1}{p}} \, \vert \, \mathcal{L}(X) = \mu, \mathcal{L}(Y) = \nu\right\}.
\end{eqnarray*}
Designating $\mu_0$ as the initial distribution of the PDMP and $\mu_t$ as its law at time $t$, we now state the main result on the PDMP in dimension one driven by \eqref{generateur}.

\begin{theo}[One dimensional PDMP]\label{theo:pdmp}  Let $p\ge1$ and denote for every $t\ge0$ $\mu_t := \mathcal{L}(X_t^{\mu_0})$ where $(X_t^{\mu_0})$ is a Markov process driven by \eqref{generateur} with initial distribution  $\mu_{0}$ (with support included in $\mathbb{R}_+^*$). If $p=1$, we have  
$$\left|\int x (\mu_0-\mu_\infty)(dx)\right| e^{-\pi t}\le \mathcal{W}_{1}(\mu_{t},\mu_{\infty})\le \mathcal{W}_{1}(\mu_{0},\mu_{\infty})e^{-{t\pi}}$$
and if $p>1$, a constant $\gamma_p$ exists such that
\begin{eqnarray*}
\mathcal{W}_{p}(\mu_{t},\mu_{\infty})&\leq & \gamma_p e^{-\frac{t\pi}{p}}.
\end{eqnarray*}
 where $(\gamma_p)_{p \geq 1}$ satisfies the recursion $\gamma_{p}^p = \gamma_{p-1}^{p-1} [p a + (1+g)^{p}]$.\end{theo}
\begin{rmq} If $p=1$, the lower and upper bounds imply the optimality of the rate obtained in the exponential. For $p>1$, the optimality of the exponent $e^{- \pi t/p}$ is still an open question.   
\end{rmq}
We now give a corollary for the limiting process that appears in Proposition \ref{prop:multiweak}.
\begin{coro}[Multi-dimensional PDMP]\label{cor:multipdmp}
Let  $(Y_t)_{t\ge0}$ be the PDMP driven by (\ref{generateur_dbras}) with initial distribution $\mu_0\in(\mathbb{R}_+^*)^d$.  Then, the conclusions of Theorem \ref{theo:pdmp} hold with $\pi=p_1-p_2$.
\end{coro}
The proof is almost obvious due to the ``tensorized'' form of the generator ${\cal L}_d$. Actually, for every starting point $y=(y_2,\ldots,y_d)$,
all the coordinates $(Y_t^i)_{t\ge0}$ are independent one-dimensional PDMPs with generator ${\cal L}$ defined by \eqref{generateur} with  
\begin{equation}\label{aibici}
a_i=\frac{1-\sigma p_1}{d-1},\quad  b_i=p_1 \quad \textnormal{and}\quad c_i=p_i/g.
\end{equation}
 The result then   easily follows from Theorem \ref{theo:pdmp} with a global rate 
given by $\min\{b_i-c_i g, i=2,\ldots,d\}= p_1-p_2$. The details are left to the reader.
\subsection{Total variation results}

When some bounds are available for the Wasserstein distance, a classical way to deduce an upper bound of the total variation is to build a two-step coupling.
 In the first step, a Wasserstein coupling is used to bring the paths sufficiently close (with a probability controlled by the Wasserstein bound). In a second step, we use a total variation coupling to try to stick the paths with a high probability. In our case, the jump size is deterministic and sticking the paths implies a non trivial coupling of the jump times. Some of the ideas to obtain the results below are in the spirit of \cite{BCG+}, who follows this strategy for the TCP process.

\begin{theo}\label{distance var totale} Let $\mu_0$ be a starting distribution with moments of any order. Then, for every $\varepsilon>0$, a $C_\varepsilon>0$ exists such that:
 $$\Vert \mu_0 P_{t}-\mu_\infty P_{t}\Vert_{TV} \leq C_\varepsilon e^{-(\alpha\pi-\varepsilon)  t} \quad\textnormal{with } \alpha=\frac{1}{2+\frac{b\pi}{ac}}.$$
\end{theo} 
Once again, this result can be extended to the multi-armed case.
\begin{coro}
Let  $(Y_t)_{t\ge0}$ be the PDMP driven by (\ref{generateur_dbras}) with initial distribution $\mu_0\in(\mathbb{R}_+^*)^d$  Then, the conclusions of Theorem \ref{distance var totale} hold with $\alpha \pi$ replaced by:
$$\sum_{i=2}^d \frac{1}{2+\frac{b_i \pi_i}{a_i c_i}} \pi_i$$
where $\pi_i=p_1-p_i$ and $a_i, b_i$ and $c_i$ are defined by \eqref{aibici}.
\end{coro}
The proof of this result is based on the remark that follows Corollary \ref{cor:multipdmp}. Owing to the ``tensorization'' property, the probability for coupling all the coordinates before time $t$ is essentially the product of the probabilities of the coupling of each coordinate. Once again, the details of this corollary are left to the reader. 

\section{Proof of the regret bound (Theorems \ref{theo:2bras1} and \ref{theo:2bras2})}\label{sec:proof_regret} 
This section is devoted to the study of the regret of the penalized two-armed bandit procedure described in Section \ref{algo:pen_LRI}. We will mainly focus 
on the proof of the explicit bound given in Theorem \ref{theo:2bras2}$(b)$ and we will give the main ideas for the proofs of Theorems \ref{theo:2bras1}  and \ref{theo:2bras1}$(a)$.  

\subsection{Notations}
In order to lighten the notations, $X_n^{1}$ will be summarized by $X_n$, so that $X_n^{2}=1-X_n$.

The proofs are then strongly based on a detailed study of the behavior of the (positive) sequence $(Y_n)_{n\ge1}$  defined by 
\begin{equation}\label{eq:defY}
\forall n \geq 1 \qquad Y_n= \frac{1-X_n}{\gamma_n}.
\end{equation}

As we said before, we will consider the following sequences $(\gamma_n)_{n \geq 1}$ and $(\rho_n)_{n \geq 1}$ below:
$$\forall n\ge1,\quad \gamma_n=\frac{\gamma_1}{\sqrt{n}}\quad\textnormal{and}\quad \rho_n=\frac{\rho_1}{\sqrt{n}} = \tilde{\rho}_1 \gamma_n \quad\textnormal{and}\quad \tilde{\rho}_1 = \frac{\rho_1}{\gamma_1},$$
where $\gamma_1$ and $\rho_1$ are constants in $(0,1)$ that will be specified later. In the meantime, we also define: $$\pi= p_1-p_2 \in (0,1).$$ 
With this setting, the pseudo-regret is
$$
\bar{R}_n = \pi  \sum_{\temps=1}^n \gamma_{\temps} \mathbb{E}[Y_{\temps}].
$$
It should be noted here that we have substituted the division by $\rho_{\temps}$ in \eqref{supnnnn} by a normalization with $\gamma_{\temps}$. This will be easier to handle in the sequel. The main issue now is to obtain a convenient upper bound for $\mathbb{E}[Y_{\temps}]$. More precisely, note that:
$$
\forall n_0 \in \mathbb{N} \quad
\forall n\le n_0-1,\quad  \bar{R}_n \le \pi n\le \pi\sqrt{n_0-1} \sqrt{n},$$
and conversely for every $n\ge n_0$,

\begin{eqnarray}\label{eq:contRNgenee}
\frac{\bar{R}_n}{\sqrt{n}} &\le&  \pi\sqrt{n_0-1}+\pi\sup_{n\ge n_0}\ES[Y_n]\frac{1}{\sqrt{n}}\sum_{n=n_0}^n\frac{\gamma_1}{\sqrt{k}} \nonumber \\
&\le &
 \pi\left(\sqrt{n_0-1}+2\gamma_1 \sup_{n\ge n_0}\ES[Y_n]\right).
\end{eqnarray}
Thus it is enough to derive an upper bound of $\ES[Y_n]$ after an iteration $n_0$ that can be on the order of $1/\pi^2$. In particular, the ``suitable" choice of $n_0$ will strongly depend on the value of $\pi$.
\subsection{Evolution of $(Y_{n})_{n\geq 1}$}
\paragraph{\tcr{Recursive dynamics of $(Y_n)_{n\ge1}$.}}
In order to understand the mechanism and difficulties of the penalized procedure, let us first roughly describe the behavior of the sequences $(X_{n})_{n \geq 1}$ and $(Y_{n})_{n \geq 1}$. \tcr{According to 
 \eqref{eq:pen_LRId3},}
\begin{eqnarray*}
\mathbb{E}\left[X_{n+1} \vert \mathcal{F}_n \right] &=& X_n + \gamma_{n+1} X_n(1-X_n) \left[p_1-p_2\right]\\
&& + \gamma_{n+1} \rho_{n+1} \left[(1-X_n)^2 (1-\sigma p_2) - X_n^2 (1-\sigma p_1) \right].
\end{eqnarray*}
It can be observed that the drift term may be split into two parts, where the main part is the usual
drift of NSa described by $h$ defined by:
\begin{equation}\label{eq:h}
\forall x\in[0,1], \quad h(x)=[p_1-p_2] x (1-x).
\end{equation}
The second term comes from the penalization procedure and depends on $\sigma$. We set
\begin{equation}\label{eq:kappa}
\kappa_{\sigma}(x)=(1-\sigma p_2) (1-x)^2 - (1 - \sigma p_1) x^2.
\end{equation}

As a consequence, we can write the evolution of $(X_n)_{n \geq 0}$ as follows:
\begin{eqnarray}\label{eq:recX}
1-X_{n+1} = 1-X_{n} - \gamma_{n+1}\left[h(X_{n}) +\rho_{n+1}\kappa_{\sigma}(X_{n})+\Delta M_{n+1}\right],
\end{eqnarray}
where $\Delta M_{n+1}$ is a martingale increment.
On the basis of the equation above, we easily derive that  
\begin{eqnarray*}
  \forall n\geq 1,\quad Y_{n+1} = Y_{n}\left(1 + \gamma_{n}(\epsilon_{n}-\pi X_{n})\right) -\rho_{n+1}\kappa(X_{n})+\Delta M_{n+1} \end{eqnarray*}
where

\begin{equation}\label{eq:epn}
\epsilon_{n} =\frac{1}{\gamma_{n+1}}- \frac{1}{\gamma_{n}} = \frac{1}{\gamma_1}\left(\sqrt{n+1}-\sqrt{n}\right)\leq \frac{1}{2\gamma_1\sqrt{n}}=\frac{\gamma_n}{2\gamma_1^2}.
\end{equation}
\noindent It follows that the increments of $(Y_{n})_{n\geq 1}$  are given by:
\begin{eqnarray*}
\Delta Y_{n+1} := Y_{n+1}-Y_{n} = \gamma_{n}\varphi_{n}(Y_{n}) - \Delta M_{n+1}
\end{eqnarray*}
where the drift function $\varphi_n$ acting on the sequence $(Y_n)_{n \geq 1}$ is defined as
\begin{eqnarray*}
\varphi_{n}(y) = \underbrace{ y\times \left[\epsilon_{n}+\pi(\gamma_{n}y-1)\right]}_{:=\varphi_n^1(y)} + \underbrace{\left(-\frac{\rho_{n+1}}{\gamma_n}\kappa_{\sigma}(1-\gamma_n y)\right)}_{:=\varphi_n^2(y)}.
\end{eqnarray*}
To better understand the underlying effects of the dynamical system, it should be recalled that the definition of the sequence $(Y_n)_{n \geq 1}$ implies that $Y_n \in [0,{\gamma_n}^{-1}]$ with ${\gamma_n}^{-1} \sim C n^{1/2}$. Since we aim to obtain a uniform bound (over $n$) of $\mathbb{E}[Y_n]$, it is thus important to understand the behavior of the drift $\varphi_n$ over $[0,{\gamma_n}^{-1}]$. \tcr{In particular,  it is of primary interest to see where the function $\varphi_n$ is negative.}

\paragraph{Crude NSa.} 


When dealing with the crude bandit algorithm (\textit{i.e.}, when $\rho_1=0$, see \eqref{algo:LRI}), the drift is reduced to $\varphi_n^1$. One can check that $\varphi_n^1(y)$ is negative \textit{iff}
$$\epsilon_{n}-\pi(1-\gamma_{n} y)<0 \Longleftrightarrow  y \leq  {\gamma_n}^{-1} -\frac{\epsilon_n}{\pi\gamma_n} \Longleftrightarrow  x \geq \frac{\epsilon_n}{\pi}$$
where $x=1-\gamma_n y$.
This means that when $x$ is close to 0 (in some sense depending on $n$, $\pi$ and $\gamma_1$), $\varphi_n^1$  becomes \tcr{positive and $Y_n$ has a tendency to increase. In others words, the dynamical system $(Y_n)_{n\ge1}$ has no \textit{mean-reverting}  when $Y_n$ is far from $0$. The fact that the crude bandit algorithm does not always converge to the good target can be understood as a consequence of this remark.}

\paragraph{Penalized and Over-Penalized NSa.}

 When the drift $\varphi_n$ contains a non zero penalty, the second term $-\varphi^2_{n}$ may help the dynamics to not be repulsive when $x$ is close to $0$, $i.e.$ when 
 $y$ is larger than $1/\gamma_n$. It can be checked that $\kappa_{\sigma}(0)=1-\sigma p_2$ and:
$$\lim_{n \rightarrow + \infty} \varphi_n\left({\gamma_n}^{-1}\right)= \frac{1}{2\gamma_1^2}-\frac{\gamma_1}{\rho_1}(1-\sigma p_2).$$
This quantity is negative under the condition:
\begin{equation}\label{eq:constraintp2}
 1- \sigma p_2> \frac{\rho_1}{2\gamma_1^3}.
\end{equation}

\tcr{But, in order to obtain a uniform bound on the regret, this constraint must be satisfied independently of $p_2$.
When $\sigma=1$, $i.e.$ in the standardly penalized case,  one remarks that for any choice of $\rho_1$ and $\gamma_1$, this is only possible if  $\rho_1/(2 \gamma_1^3) > 1-p_2$. At this stage, one can thus understand
the over-penalization as a way of controlling uniformly (in $p_2$) the negativity of $\varphi_n$ far from $y=0$ (see Figure \ref{fig:rappel2}).}\smallskip

\noindent \tcr{In view of the main results, there are still two problems. The first one is that even in the over-penalized case, Inequality \eqref{eq:constraintp2} implies some constraints on $\gamma_1$ and $\rho_1$, which do not appear in Theorem \ref{theo:2bras2}.
The second one which is more embarassing for the study of $(\ES[Y_n])_{n\ge1}$ is that, near $y=0$, $\varphi_n$ is positive since $\varphi_n(0)=1-\sigma p_1$ ((see Figure \ref{fig:rappel2})). This repulsive behavior near $y=0$ can be understood as the 
	counterpart induced by the penalization.
In order to bypass the two previous problems, the main argument will be the \textit{increase of exponent} (see next section) where we show that we can replace the study of $(\ES[Y_n])_{n\ge1}$ by the one of a sequence which both has a nicer behavior near $y=0$ and alleviates the constraint \eqref{eq:constraintp2}.}





\begin{figure}[h!] 
\begin{center}
\includegraphics[width=10cm,height=4cm]{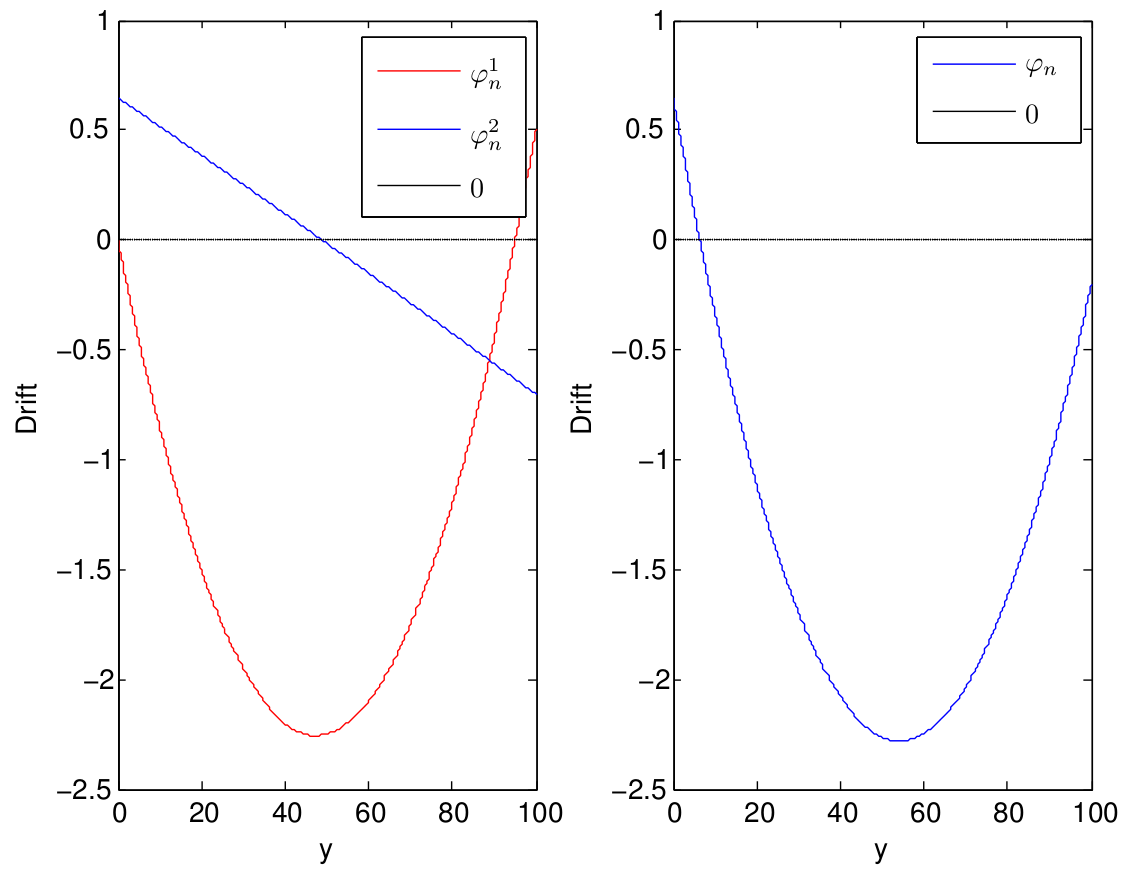}
\end{center}
\caption{Drift decomposition (left) and global (right) when $y \in [0,\frac{1}{\gamma_n}]$ with $\gamma_1 = \rho_1 = 1$, $p_{1}=0.7$ ,$p_{2}=0.6$, $\sigma=0.5$.\label{fig:rappel2}}  
\end{figure}


 \subsection{Increase of exponent}
\noindent
We introduce the sequence $(Z_{n}^{(r)})_{n\geq 0}$ defined by:

\begin{equation}\label{eq:znr}
\forall n \geq 1 \qquad Z_{n}^{(r)} = \frac{(1-X_{n})^{r}}{\gamma_{n}}.\end{equation}

\tcr{At this stage, one can first remark that $a.s.$, for every $r\ge1$,  $Z_n^{(r)}\le Z_n^{(1)}=Y_n$. One can thus guess that the difficulties tackled at the end of the previous section will be easier to overcome for $(\ES[Z_n^{(r)}])_{n\ge1}$ with $r>1$.
Of course, this remark has an interest if conversely, one is able to relate the control of $\ES[Y_n]$ to those of $\ES[Z_n^{(r)}]$, $r\ge1$.}

This is the purpose of Proposition \ref{prop:increment} where taking advantage of the structure of the algorithm, one shows that for every $r\ge1$, $\mathbb{E}[\znr]$  can be controlled by a function of $\mathbb{E}[\znrr]$.


\noindent Let us define the bounded function $h_r$ on $[0,1]$:
\begin{equation}\label{eq:hr}
\forall \gamma \in [0,1] \qquad 
h_r(\gamma)=\frac{(1+\gamma)^r-1-r\gamma}{r \gamma^2}.
\end{equation}

\begin{proposition}\label{prop:increment} 
Let $r\in\mathbb{N}^*$, $\gamma_1\in(0, 1)$ and  $0<\epsilon \le \epsilon_0 = \frac{1}{3}$, and set
\begin{equation}\label{eq:n0}
n_0(\epsilon,\pi,\gamma_1):= \left\lfloor \frac{1}{4 \epsilon^2 \gamma^2_1 \pi^2}\right\rfloor+1.
\end{equation}
Then, if $2\epsilon\gamma_1^2(r-\varepsilon)\le 1$,
 $$
\sup_{n \geq n_0} \mathbb{E} \znr \leq \mathbb{E} \znzr + 
\frac{r}{\pi(r-\epsilon)} \left[\tilde{\rho}_1 + h_r(\gamma_{n_0})+ \pi \sup_{n \geq n_0}\ES[ \znrr] \right].$$
In particular, for $r=1,2$, the previous inequality holds  for every $\gamma_1\in(0,1)$ and $\varepsilon\in(0,1/3]$.
\end{proposition}
\begin{rmq}  
\tcr{Note that the above result induces some constraints on $\epsilon$ and $\gamma$. These constraints, which allow us to manage the constants of the inequality, are mainly adapted to the proof of
Theorem \ref{theo:2bras2} $(b)$. In fact, in the proofs of Theorems \ref{theo:2bras1} and \ref{theo:2bras2} $(a)$, we will need  to rewrite the above property in a slightly different way (see Section \ref{subsec:prooftheo1} for details).}  
\end{rmq}

\subparagraph{Proof} 
For any integer $r>0$ and $ n\geq 0$, the binomial formula applied to \eqref{eq:recX} leads to
\begin{eqnarray*}
(1-X_{n+1})^{r }&=& \left(1-X_{n} - \Delta X_{n+1} \right)^r \\
& = & (1-X_{n})^{r}- r (1-X_{n})^{r-1} \Delta X_{n+1}\\&& +\sum_{j=0}^{r-2}\binom{r}{j}(1-X_{n})^{j}(-\Delta X_{n+1})^{r-j}.
\end{eqnarray*}

where  $\sum_\emptyset =0$ and $\Delta X_{n+1}=X_{n+1}-X_n = \gamma_{n+1} [
h(X_n)+\rho_{n+1}\kappa_{\sigma}(X_n) + \Delta M_{n+1}]
$.
From the definition of $h$ given in \eqref{eq:h}, we get
$$
(1-x)^{r-1}  [h(x)+\rho_{n+1}\kappa_{\sigma}(x)] =  \pi x(1-x)^{r} + \rho_{n+1} \kappa_{\sigma}(x)(1-x)^{r-1}.
$$
If we define now
\begin{equation}\label{eq:betanr}
\beta_{n}^{(r)} = -r \rho_{n+1}(1-X_{n})^{r-1}\kappa_{\sigma}(X_{n}) + \frac{1}{\gamma_{n+1}}\sum_{j=0}^{r-2}\binom{r}{j}(1-X_{n})^{j}(-\Delta X_{n+1})^{r-j},
\end{equation} we can then conclude using \eqref{eq:znr} that
\begin{eqnarray}
\lefteqn{\znnr} \nonumber \\ &=&\znr \frac{\gamma_n}{\gamma_{n+1}} - \gamma_n r \pi X_n \znr
+ \beta_{n}^{(r)}- r(1-X_{n})^{r-1}\Delta M_{n+1}\nonumber \\
&=&\znr \left(1 + \gamma_n \left[ \frac{1}{\gamma_{n+1}} - \frac{1}{\gamma_n} -  r \pi X_n \right]\right) + \beta_{n}^{(r)}- r(1-X_{n})^{r-1}\Delta M_{n+1} \nonumber\\
& = & \znr \left(1 + \gamma_n \left[ \epsilon_n -  r \pi X_n \right]\right) + \beta_{n}^{(r)}- r(1-X_{n})^{r-1}\Delta M_{n+1}\nonumber\\
& = & \znr \left(1 + \gamma_n \left[ \epsilon_n -  r \pi \right]\right) 
+r \pi \gamma_n (1-X_n) \znr  + \beta_{n}^{(r)}- r(1-X_{n})^{r-1}\Delta M_{n+1}\nonumber\\
& = & \znr \left(1 + \gamma_n \left[ \epsilon_n -  r \pi \right]\right) 
+r \pi \gamma_n \znrr  + \beta_{n}^{(r)}- r(1-X_{n})^{r-1}\Delta M_{n+1}.
\label{eq:recZ}
\end{eqnarray}
The formulation above is important: it exhibits a contraction of
 $\left(1 + \gamma_n \left[ \epsilon_n -  r \pi \right]\right)$ on $\znr$ that can be used jointly with an upper bound of $\znrr$ and a simple majorization of $\beta_n^{(r)}$.
In this view, we study \eqref{eq:betanr}: $|\Delta X_{n+1}|\leq \gamma_{n+1}$ $a.s.$ and \eqref{eq:kappa} yields $|\kappa_\sigma(x)|\leq (1-\sigma p_2)$. Now, with $h_r$ given in \eqref{eq:hr}, we get

$$\beta_{n}^{(r)} \leq r \tilde{\rho}_1\gamma_{n+1}+\sum_{j=0}^{r-2}\binom{r}{j}(\gamma_{n+1})^{r-j-1}\leq r\left(\tilde{\rho}_1 +h_r(\gamma_{n+1})\right)\gamma_{n+1}.$$

For any $\epsilon \in (0,1)$, we can see in \eqref{eq:recZ} that the contraction coefficient
 can be useful as soon as $n$ is large enough. More precisely, using \eqref{eq:epn}, we see that
 $$
 \epsilon_n \leq \epsilon \Longleftarrow n \geq n_{0}(\epsilon,\pi,\gamma_1) := \left\lfloor \frac{1}{4 \epsilon^2 \gamma^2_1 \pi^2}\right\rfloor+1.
 $$
 Then, for every $n\ge n_0(\epsilon,\pi,\gamma_1)$, 
$$1 + \gamma_n \left[ \epsilon_n -  r \pi \right]\le 1-\alpha_r \gamma_n\quad\textnormal{with}\quad \alpha_r=  \pi(r-\varepsilon).$$
In the sequel, we will omit the dependence of $n_0$ in $(\epsilon,\pi,\gamma_1)$ and will just use the notation $n_0$.
Also remark that under the condition  $2\epsilon\gamma_1^2(r-\varepsilon)\le 1$, we have
$\alpha_r\gamma_j<1$ for every $\pi\in(0,1)$ and for every $j\ge n_0$ (one can in particular check that $2\epsilon\gamma_1^2(r-\varepsilon)\le 1$ is true for every $\epsilon\in(0,1/3)$ and $\gamma_1\in(0,1)$ if $r=1,2$). Thus, by a simple recursion based on \eqref{eq:recZ}, one obtains for every $n\ge n_0+1$,
\begin{eqnarray*}
\mathbb{E}(Z_{n}^{(r)}) \leq\mathbb{E}( Z_{n_0}^{(r)})\prod\limits_{j=n_0}^{n-1}(1- \alpha_r\gamma_{j}) +  \sum\limits_{j=n_0}^{n-1}\left(r\pi \gamma_{j}\mathbb{E}(Z_{j}^{(r+1)})+\beta_{j}^{(r)}\right)\prod\limits_{l=j}^{n-1}(1-\alpha_r\gamma_{l})
\end{eqnarray*}
If we call $\Theta_{r}= r\left(\pi\sup\limits_{j\geq n_0}\left(\mathbb{E}(Z_{j}^{(r+1)}) + \tilde{\rho}_1+h_r(\gamma_j)\right)\right)$, an iteration of the previous inequality yields: 
\begin{eqnarray*}
\mathbb{E}(Z_{n}^{(r)}) \leq \mathbb{E}( Z_{n_0}^{(r)})+ \Theta_{r}\sum\limits_{j=n_0}^{n-1}\gamma_{j}\prod\limits_{l=j}^{n-1}(1-\alpha_r \gamma_{l}).
\end{eqnarray*}

We aim to apply Lemma \ref{lemma:sum} (deferred to the appendix section) to the last term. It is possible as soon as
$$
n_0 \geq  \frac{1}{(\alpha_r \gamma_1)^2}.
$$

This last condition is fulfilled for any $r \geq 1$ when
$
\frac{1}{4 \epsilon^2 \gamma_1^2 \pi^2} \geq \frac{1}{ (1-\epsilon)^2 \pi^2 \gamma_1^2},
$
$i.e.$ when $\epsilon\le 1/3$.

Then, by Lemma \ref{lemma:sum}, one deduces that $\forall \epsilon \leq 1/3$ and $
\forall n \geq n_0$ :
$$
\sup_{n \geq n_0} \mathbb{E} \znr \leq \mathbb{E} \znzr + 
\frac{r}{\pi(r-\epsilon)} \left[\tilde{\rho}_1 + h_r(\gamma_{n_0})+ \pi \sup_{n \geq n_0} \znrr \right].
$$
\hfill \cqfd

On the basis of the last proposition and a recursive argument, we can now deduce the following key observations.
\begin{coro}\label{coro:boundY}
Assume that $\epsilon \in(0, 1/3)$, $\gamma_1\in(0,1)$ and that $n_0$ is defined in \eqref{eq:n0}. Then,
\begin{eqnarray}\label{eq:borneY}
\sup_{n \geq n_0} \mathbb{E} [Y_n]& \leq &\mathbb{E}[ Z^{(1)}_{n_0}] + \frac{\ES[Z_{n_0}^{(2)}]}{1-\epsilon} + \frac{1}{\pi(1-\epsilon)} \left[\tilde{\rho}_1+\frac{\tilde{\rho}_1}{1-\epsilon/2} +\frac{1}{2(1-\epsilon/2)}\right] \nonumber \\
&& + \frac{ \sup_{n \geq n_0} \mathbb{E} Z_n^{(3)} }{(1-\epsilon)(1-\epsilon/2)}.
\end{eqnarray}

\end{coro}
\begin{rmq}\label{rmq:coro2}
As in Proposition \ref{prop:increment}, this property is mainly written in view of Theorem \ref{theo:2bras2} $(b)$ where we only need to use the increase of exponent for $r=1,2$. For Theorems \ref{theo:2bras1} and \ref{theo:2bras2} $(a)$ with $\sigma\in(0,1)$, we will need to use it for large values of $r$.

\end{rmq}

\subsection{Bound for $(\mathbb{E}(Z_{n}^{(3)}))_{n\geq n_0}$}

 As seen in Corollary \ref{coro:boundY}, our next task is to bound $\mathbb{E}(Z_{n}^{(3)})$ for $n\geq n_0$ to obtain a tractable application of Equation \eqref{eq:borneY}. Such a bound is reached through careful inspection of the increments $\Delta Z_{n+1}^{(3)}:=Z_{n+1}^{(3)}-Z_{n}^{(3)}$.
 
\begin{lemma}[Decomposition of $Z_{n}^{(3)}$]\label{lemme:decompzn3}
 For every $n\ge1$,
\begin{equation*}
\ES[\Delta Z_{n+1}^{(3)}|{\cal F}_n]=\gamma_{n+1}(1-X_{n})P_{n}(X_{n}) + \Delta R_{n},
\end{equation*}
where for every $n\in\EN$, $P_{n}$ is a polynomial function defined by 
\begin{align}
P_n(x)&= \frac{(1-x)^2}{\gamma_{n+1}}(\varepsilon_n-3\pi x)-3 \tilde{\rho}_1(1-x)\kappa_{\sigma}(x) 
+3\left(x(1-x)^2  p_1+x^2(1-x)p_2\right)\nonumber\\& +\gamma_{n+1}\left(-x(1-x)^2p_1+x^3p_2\right),\label{eq:defPN}
\end{align}
and if $\gamma_1$ and $n_0$ satisfy the assumptions of Proposition \ref{prop:increment}, then 
\begin{equation*}
\forall n\ge n_0,\quad \Delta R_{n}\leq (1-\sigma p_2) \left[ 3 \gamma_{n+1} \rho_{n+1}^2 + \gamma_{n+1}^2 \rho_{n+1}^3 \right].
\end{equation*}

\end{lemma}
\begin{rmq}
\begin{itemize}
\item 
 The keypoint is that $\gamma_k=\gamma_1 k^{-1/2}$ and, therefore, the series $\sum_{n\ge1} \Delta R_k$ is uniformly bounded, regardless of the value of $\pi$. This will be enough to obtain a competitive upper bound of the regret.
With the choice of $n_0$ given in \eqref{eq:n0}, careful inspection of Lemma \ref{lemme:decompzn3} leads to:
\begin{equation}\label{eq:boundRN}
\sum_{k \geq n_0} \Delta R_k \leq  12 \gamma_1^{4} \tilde{\rho}_1^2 \epsilon \pi 
+ \frac{16}{3} \gamma_1^{8} \tilde{\rho}_1^3 \epsilon^3 \pi^3. \end{equation}
 \item 
As in Remark \ref{rmq:coro2}, it should be noted that  for Theorems \ref{theo:2bras1} and \ref{theo:2bras2} $(a)$ with $\sigma\in(0,1)$, we will need to use such a development with some larger values of $r$ (see the end of this section for details). 
\end{itemize}
\end{rmq}
\begin{proof} We again use Equation (\ref{eq:recZ}) and deduce that:
\begin{align}
Z_{n+1}^{(3)}-Z_{n}^{(3)}= (1-X_n)^3(\epsilon_{n}-3\pi X_{n})) -3 \tilde{\rho}_1\gamma_{n+1}(1-X_{n})^{2}\kappa_{\sigma}(X_{n})\label{eq:recz31}\\
 + \frac{1}{\gamma_{n+1}}\sum_{j=0}^{1}\binom{3}{j}(1-X_{n})^{j}(-\Delta X_{n+1})^{3-j}
- 3(1-X_{n})^{2}\Delta M_{n+1}\label{eq:recz32}
\end{align}
First, note that terms in Equation \eqref{eq:recz31} are associated with the first two terms in the definition of $P_n$ introduced in \eqref{eq:defPN}
up to a multiplication by $(1-X_n) \gamma_{n+1}$.

 Second, we can easily compute the expectations involved in the sum of Equation \eqref{eq:recz32} since the events are all disjointed. On the one hand, when $j=1$ we have
\begin{align*}
&\frac{1}{\gamma_{n+1}}\ES[(-\Delta X_{n+1})^{2}|{\cal F}_n]\,=\, \gamma_{n+1}\sigma \left(  p_1 X_n(1-X_n)^2+ p_2 (1-X_n)X_n^2\right)\\
&+\gamma_{n+1}(1-\sigma)\left(  p_1 X_n(1-X_n-\rho_{n+1} X_n)^2+  p_2 (1-X_n)(X_n-\rho_{n+1} (1-X_n))^2\right)\\
&+\gamma_{n+1} \rho_{n+1}^2 \left[ X_n^3(1- p_1) + (1-X_n)^3 (1- p_2)\right].
\end{align*}
Further computations yield: 
 \begin{eqnarray*}
\frac{1}{\gamma_{n+1}}\ES[(-\Delta X_{n+1})^{2}|{\cal F}_n]&=&  \gamma_{n+1}X_n\left( \sigma p_1 (1-X_n)^2+ \sigma p_2 X_n^2\right)+\Delta A_n^{(1)}\nonumber\\
\hspace{-2cm}+\lefteqn{
\underbrace{\gamma_{n+1} \rho_{n+1}^2 \left[ X_n^3(1-\sigma p_1) + (1-X_n)^3 (1-\sigma p_2)\right]}_{:= \Delta R_{n}^{(1)}}}
 \end{eqnarray*}
with $\Delta A_n^{(1)}=-2\rho_{n+1}\gamma_{n+1}X_n(1-X_n)(1-\sigma)(X_n p_1+(1-X_n) p_2)$. On the other hand, we can also compute the term when $j=0$:
 \begin{eqnarray*}
  \frac{1}{\gamma_{n+1}}\ES[(-\Delta X_{n+1})^{3}|{\cal F}_n]&= &
  \gamma_{n+1}^2 X_n (1-X_{n})\left(p_2 X_n^2-p_1 (1-X_n)^2  \right)\nonumber\\
  +\Delta A_n^{(2)}+ \lefteqn{\underbrace{\gamma_{n+1}^2 \rho_{n+1}^3 \left[ X_n^4(1-\sigma p_1) - (1-X_n)^4 (1-\sigma p_2)\right]}_{:= \Delta R_{n}^{(2)}}}
  \end{eqnarray*}
with $  \Delta A_n^{(2)}\le 3\gamma_{n+1}^2\rho_{n+1}(1-\sigma)X_n(1-X_n)^2\left(\pi X_n+\rho_{n+1}(1-X_n) p_2\right)$. Set
$\Delta R_n^{(3)}=(1-X_n)\Delta A_n^{(1)}+\Delta A_n^{(2)}$ and $\Delta R_n := 3 (1-X_n) \Delta R_n^{(1)} +   \Delta R_n^{(2)}$. Plugging the previous controls into 
\eqref{eq:recz32} yields
\begin{equation}\label{eq:truc}
\ES[\Delta Z_{n+1}^{(3)}|{\cal F}_n]\le\gamma_{n+1}(1-X_{n})P_{n}(X_{n})+\Delta R_{n}.
\end{equation}

Note that $\Delta R_n^{(1)}$ can be upper bounded as follows:
$$
3 (1-X_n) \Delta R_n^{(1)} \leq 3 \gamma_{n+1} \rho_{n+1}^2 (1-\sigma p_2) \max_{0 \leq t \leq 1} \left[ \frac{1-\sigma p_1}{1-\sigma p_2} t^3(1-t) + (1-t)^4 \right].
$$
Since $1-\sigma p_1 \leq 1-\sigma p_2$, a study of the function  shows that $a t^3(1-t) + (1-t)^4$ when $a \in (0,1)$ reaches its maximal value for $t=0$. This 
 leads to: $$3 (1-X_n) \Delta R_n^{(1)} \leq 3 \gamma_{n+1} \rho_{n+1}^2 (1-\sigma p_2).$$
For $\Delta R_n^{(2)}$, we have $\Delta R_n^{(2)} \leq \gamma_{n+1}^2 \rho_{n+1}^3 (1-\sigma p_2) \max_{0\leq t \leq 1} \left[\frac{1-\sigma p_1}{1-\sigma p_2} t^4 - (1-t)^4 \right],$
which involves an increasing function of $t$. Thus, we have
$$\Delta R_n^{(2)} \leq \gamma_{n+1}^2 \rho_{n+1}^3 (1-\sigma p_1) \leq \gamma_{n+1}^2 \rho_{n+1}^3 (1-\sigma p_2).$$
Finally, if $\gamma_1$ and  $n_0$ satisfy the assumptions of Proposition \ref{prop:increment}, then for every $n\ge n_0$,
$\gamma_n\le 2/3$ and it follows that $\Delta R_n^{(3)}\le 0.$
The result follows according to Equation \eqref{eq:truc}.
\end{proof} 

 In order to bound $\sup_{n \geq n_0} \mathbb{E}( Z_{n}^{(3)})$, we now have  to precisely study the polynomial function $P_{n}$ and exhibit a mean reverting effect on its dynamics. 
\begin{proposition}\label{prop:etudePn}
Let $\epsilon \in(0, \frac{1}{3})$, $\tilde{\rho}_1 \leq   \frac{227}{232}$ and 
$\frac{1}{3\sqrt{2}(1-\sigma)\tilde{\rho}_1} \leq \gamma_1^2 \leq \frac{3}{2(1+\tilde{\rho}_1) }$. Then
\begin{itemize}
\item[$i)$] The polynomial $P_n$ given by \eqref{eq:defPN} is negative on $[0, 1 - \frac{2(1+\tilde{\rho}_1)}{\pi}{\gamma_{n+1}}]$.
\item[$ii)$] $Z_n^{(3)}$ satisfies 
\begin{eqnarray*}
\sup_{n \geq n_0} \mathbb{E} Z_n^{(3)} &\leq &\mathbb{E} Z_{n_0}^{(3)}+ 12 \gamma_1^{4} \tilde{\rho}_1^2 \epsilon \pi 
+ \frac{16}{3} \gamma_1^{8} \tilde{\rho}_1^3 \epsilon^3 \pi^3\\
&  &+ \frac{8 \gamma_1^4 \epsilon(1+\tilde{\rho}_1) \left[ 1+(1+\tilde{\rho}_1) [2+ 6  \tilde{\rho}_1  + 12 \gamma_1^2 \epsilon  ]
\right]}{\pi}.
\end{eqnarray*}

\end{itemize}
\end{proposition}
\begin{rmq} The above result is given under some technical conditions that will lead to a sharp explicit bound. Nevertheless, the reader has to keep in mind that in view of the condition on $\sigma$, the ``universal'' bound on $(\mathbb{E}( Z_{n}^{(3)}))_{n\ge n_0}$ is only accessible when $\sigma<1$, $i.e.$ in the over-penalized case. When $\sigma=1$, some bounds will be attainable only if $p_2$ is not too large (see \eqref{eq:constraintp2} for a similar statement when $r=1$), and in order to
alleviate the constraint on $p_2$, it will ne necessary to take a larger exponent than $r=3$ (see Subsection \ref{subsec:prooftheo1} for details).
\end{rmq}

\begin{proof} We first provide the proof of $i)$.
The function $P_n$ introduced in \eqref{eq:defPN} is a third degree polynomial and for $n \geq n_0$:
\begin{eqnarray*}
P_n(0)&=&\frac{\epsilon_n}{\gamma_{n+1}} - 3 \tilde{\rho}_1 \kappa_{\sigma}(0)\\
 &\leq& \frac{\gamma_n}{2 \gamma_1^2 \gamma_{n+1}} - 3 \tilde{\rho}_1 (1 - \sigma p_2)\\
 & \leq & \frac{\sqrt{1+n_0^{-1}}}{2 \gamma_1^2} - 3 \tilde{\rho}_1 (1-\sigma p_2) \\
\end{eqnarray*}
Since $p_2<1$, this last quantity is negative if:
\begin{equation}\label{eq:cond1} 
\rho_1 \gamma_1 \geq \frac{1}{3 \sqrt{2} (1-\sigma)}.
\end{equation}
In a same way, we can check that $P_n(1)=\gamma_{n+1} p_2 >0$ and, therefore, $P_n$ has one root in the interval $(0,1)$.
Careful inspection of the leading coefficient (designated $a_n x^3$) of $P_n$ in \eqref{eq:defPN} shows that:
$$
a_n=\left[3(1+\sigma \tilde{\rho}_1) - \frac{3}{\gamma_{n+1}} - \gamma_{n+1}\right]\pi.
$$
The leading coefficient $a_n$ is negative as soon as 
$3(1+\sigma \tilde{\rho}_1) \leq \frac{3}{\gamma_{n+1}}$. Again, the choice of $n_0$ in \eqref{eq:n0} shows that this last condition is fulfilled as soon as
\begin{equation}\label{eq:cond2}
\frac{1}{\epsilon} \geq 2 \gamma_1 \pi (\gamma_1+\sigma \rho_1).
\end{equation}
It should however be noted that we have assumed $\epsilon \in (0,1/3]$ so that $\frac{1}{\epsilon} \geq 3$. As a consequence, \eqref{eq:cond1} and \eqref{eq:cond2} are satisfied as soon as $(\gamma_1,\tilde{\rho_1})$ satisfies
$$
\frac{1}{3\sqrt{2}(1-\sigma)\tilde{\rho}_1} \leq \gamma_1^2 \leq \frac{3}{2(1+\tilde{\rho}_1)}
$$

Hence, if \eqref{eq:cond1} and \eqref{eq:cond2} hold, $P_n$ possesses one root in $(-\infty,0)$ and another one in $(1,+\infty)$. Consequently, $P_n$ has a unique root in $(0,1)$. We now consider:
\begin{equation*}
\xi_n=  \frac{2 (1+\tilde{\rho}_1) }{\pi} \gamma_{n+1} := \xi \gamma_{n+1}. \end{equation*} 
We compute that:
\begin{eqnarray*}
\lefteqn{
P_n(1-\xi_n)}\\
& =& \frac{\xi_n^2}{\gamma_{n+1}} (\epsilon_n - 3 \pi(1-\xi_n))-3\tilde{\rho}_1 \xi_n \left[(1-\sigma p_2) \xi_n^2 -(1-\sigma p_1) (1-\xi_n)^2\right]\\
&& + 3 \left[ (1-\xi_n) \xi_n^2 p_1 +(1-\xi_n)^2 \xi_n p_2 \right] + \gamma_{n+1} \left[ (1-\xi_n)^3 p_2 - \xi_n^2 (1-\xi_n) \right].
\end{eqnarray*}
Hence, replacing $\xi_n$ by $\xi \gamma_{n+1}$ and simplifying by $\gamma_{n+1}$, we see that $P_n(1-\xi_n)$ is negative when
\begin{align*}
\overbrace{ \frac{\xi^2 \epsilon_n}{(1-\xi_n)} + 3 \tilde{\rho}_1  (1-\sigma p_1)(1-\xi_n)\xi + 3 p_1 \gamma_{n+1} \xi^2 + 3 p_2 (1-\xi_n)  \xi +  p_2 (1-\xi_n)^2}^{:=A_n(\xi)}&\\
\leq   \underbrace{3 \pi \xi^2+\frac{3 \tilde{\rho}_1 \gamma_{n+1}^2 \xi^3 (1-\sigma p_2)}{1-\xi_n} + \gamma_{n+1}^2\xi^2}_{:=B_n(\xi)}.
\end{align*}
From \eqref{eq:epn}, we know that $\epsilon_n \leq \frac{\gamma_{n+1}}{2\gamma_1^2}$, and $1 -\xi_n \leq 1$ thus 

$$
A_n(\xi) \leq \xi^{2} \gamma_{n+1} \left( \frac{1}{2 \gamma_1^2 (1-\xi_n)}+3p_1	 \right) + 3 \xi \left( \tilde{\rho}_1+1 \right) + 1
$$
In the meantime, we will use the simple lower bound $B_n(\xi) \geq 3 \pi  \xi^2$. We can check that $1-\xi_n = 1 -  \frac{2(1+\tilde{\rho}_1) \gamma_{n+1}}{\pi} \geq 1 - 4 \epsilon (1+\tilde{\rho}_1) \gamma_1^2$ since $\gamma_{n_0} \leq  2 \epsilon \gamma_1^2 \pi$. Thus
\begin{eqnarray*}
\lefteqn{
A_n\left(\frac{2(1+\tilde{\rho}_1)}{\pi}\right)}\\
 &\leq& \frac{4 (1+\tilde{\rho}_1)^2}{\pi^2} \gamma_{n+1} \left[3 p_1+\frac{1}{2 \gamma_1^2 \left[1- 4 \epsilon (1+\tilde{\rho}_1) \gamma_1^2\right]}\right] + \frac{6 (1+\tilde{\rho}_1)^2}{\pi}+1 \\
& \leq & \frac{(1+\tilde{\rho}_1)^2}{\pi} \left[ 24 \epsilon \gamma_1^2 p_1+ \frac{4\epsilon}{1- 4 \epsilon (1+\tilde{\rho}_1) \gamma_1^2}+ 7
\right]
\end{eqnarray*}
and
$$
B_n\left(\frac{2(1+\tilde{\rho}_1)}{\pi}\right) \geq \frac{12 (1+\tilde{\rho}_1)^2}{\pi}.
$$
As a consequence, $P_n(1-\xi_n)$ is negative if we have
$$
5 \geq  24 \epsilon \gamma_1^2 + \frac{4\epsilon}{1- 4 \epsilon (1+\tilde{\rho}_1) \gamma_1^2}
$$
From the constraint on $\gamma_1$, another computation shows that the above condition is fulfilled when
$
\epsilon^2  \frac{128 (1+\tilde{\rho}_1)}{3} - \epsilon [84 + 40 (1+\tilde{\rho}_1)] + 45 \geq 0.
$
We then observe that all values of $\epsilon$ in $(0, \frac{1}{3}]$ can be conveniently used when
$\tilde{\rho}_1 \leq \frac{227}{232}$.\\

To obtain $ii)$, the main idea is to use the sharp estimation of the sign of $P_n$ on $[0,1]$ and to obtain an upper bound of $\mathbb{E} Z_{n}^{(3)}$. Note that:
\begin{eqnarray*}
\lefteqn{\sup_{0 \leq t \leq 1} \gamma_{n+1} (1-t) P_n(t) }&&\\&=&  \gamma_{n+1} \sup_{1-\xi_n \leq t \leq 1} (1-t) P_n(t) \\
& = &  \gamma_{n+1} \sup_{1-\xi_n \leq t \leq 1}\LARGE\left\{(1-t)^3 \left[\epsilon_n - 3 \pi t\right] - 3 \tilde{\rho}_1 (1-t)^{2} \kappa_\sigma(t) \right.\\&&\left.+ 3 \left[ t (1-t)^3 p_1 + t^2 (1-t)^2 p_2\right] + \gamma_{n+1} \left[ - t (1-t)^3 p_1+t^3(1-t)p_2\right]\LARGE\right\}
\end{eqnarray*}
We have seen in the proof of $i)$ that $t \in [1-\xi_n,1] \Longrightarrow  \epsilon_n \leq 3 \pi t$.
Hence, using $\kappa_{\sigma}(t) \geq -(1-\sigma p_1) t^2$, we have 
\begin{eqnarray*}
\lefteqn{\sup_{0 \leq t \leq 1} \gamma_{n+1} (1-t) P_n(t) }&&\\&\leq & \gamma_{n+1}\left[  3 \tilde{\rho}_1 (1-\sigma p_1) \xi_n^2 + 3 p_1 \xi_n^3 + p_2 \xi_n^2+\gamma_{n+1} \xi_n\right] \\
& \leq &  \frac{C_1(\tilde{\rho}_1,p_1,p_2,\sigma)}{\pi^2} \gamma_{n+1}^3 + \frac{C_2(\tilde{\rho}_1,p_1)}{\pi^3} \gamma_{n+1}^4
\end{eqnarray*}
with $C_1(\tilde{\rho}_1,p_1,p_2,\sigma) = (1+\tilde{\rho}_1) \left( 12 \tilde{\rho}_1 (1+\tilde{\rho}_1)(1-\sigma p_1)+ 4p_2(1+\tilde{\rho}_1) +2 \pi\right) $ and $C_2(\tilde{\rho}_1,p_1) = 24 p_1  (1+\tilde{\rho}_1)^3$ shortenned in $C_1$ and $C_2$ below.
We apply Lemma \ref{lemme:decompzn3} to upper bound $ \sup_{n \geq n_0} \mathbb{E} Z_{n}^{(3)}$:
\begin{eqnarray*}
\lefteqn{\sup_{n \geq n_0} \mathbb{E} Z_n^{(3)}} &&\\
& \leq& \mathbb{E} Z_{n_0}^{(3)} + \displaystyle \sup_{n \geq n_0} \mathbb{E}  \sum_{k=n_0}^{n} \Delta Z_{n+1}^{(3)} \\
& \leq &  \mathbb{E} Z_{n_0}^{(3)} + \displaystyle \sup_{n \geq n_0} \mathbb{E}\left[  \sum_{k=n_0}^{n} \gamma_{k+1} (1-X_k) P_k(X_k)+\Delta R_k\right] \\
& \leq &  \mathbb{E} Z_{n_0}^{(3)}+\frac{C_1}{\pi^2} \sum_{k=n_0}^{\infty} \gamma_{k+1}^3+\frac{C_2}{\pi^3} \sum_{k=n_0}^{\infty} \gamma_{k+1}^4 + \sum_{k=n_0}^{\infty} \ES\Delta R_k
\end{eqnarray*}

Using a simple comparison argument with the integrals $\int_{n_0}^{\infty} t^{-\alpha}dt$, we obtain: 
$$\sum_{k=n_0}^{\infty} \gamma_{k+1}^3 \leq 2 \gamma_1^3 n_0^{-1/2} \leq 4 \gamma_1^4 \epsilon \pi \qquad \text{and}\qquad 
\sum_{k=n_0}^{\infty} \gamma_{k+1}^4 \leq  \gamma_1^4 n_0^{-1} \leq 4 \gamma_1^6 \epsilon^2 \pi^2.$$
We then deduce that:
$$
\sup_{n \geq n_0} \mathbb{E} Z_n^{(3)} \leq \mathbb{E} Z_{n_0}^{(3)}+ \frac{4 \gamma_1^4 \epsilon C_1}{\pi}+ \frac{4 \gamma_1^6 \epsilon^2 C_2}{\pi}+ \sum_{k=n_0}^{\infty} \Delta R_k.
$$
The result now follows using \eqref{eq:boundRN}.
\end{proof}

\paragraph{Explicit bound.}

We can now conclude  the proof of Theorem \ref{theo:2bras2}.
\begin{proof}[Proof  of Theorem \ref{theo:2bras2} $(b)$]
 We consider the extreme  over-penalized case obtained with $\sigma =0$.
and use a power increment until $r=3$. Recall that $n_0:=n_0(\epsilon,\pi,\gamma_1)$ is defined by \eqref{eq:n0}. In particular,
$\sqrt{n_0-1}\le (2\epsilon\gamma_1\pi)^{-1}$ and for $i=1,2,3$, $\pi \ES[Z_{n_0}^{(i)}]\le (2\epsilon\gamma_1^2)^{-1}+(\gamma_1)^{-1}$.  Taking the results of
Proposition \ref{prop:etudePn} $ii)$ and Corollary \ref{coro:boundY} and plugging them into \eqref{eq:contRNgenee}, a series of computations yields:
\begin{eqnarray}\label{eq:optimC}
\frac{\sup_{p_1 \geq p_2} \bar{R}_n}{\sqrt{n}}\le {c(\gamma_1,\tilde{\rho}_1,\varepsilon)}:= T_1(\gamma_1,
\tilde{\rho}_1,\epsilon) +\frac{2\gamma_1}{(1-\epsilon)(1-\epsilon/2)}T_2(\gamma_1,\tilde{\rho}_1,\epsilon),
\end{eqnarray}
where 
\begin{align*}
T_1(\gamma_1,\tilde{\rho}_1,\epsilon) &=\frac{1}{2\epsilon\gamma_1}+\left(\frac{1}{\epsilon\gamma_1}+2\right)\left(1+\frac{1}{1-\epsilon}+\frac{1}{(1-\epsilon)(1-\epsilon/2)}\right)\\
&+2\rho_1\left(\frac{1}{1-\epsilon}+\frac{1}{(1-\epsilon)(1-\epsilon/2)}\right)+\frac{\gamma_1}{(1-\epsilon)(1-\epsilon/2)},
\end{align*}
and 
\begin{eqnarray*}
\lefteqn{
T_2(\gamma_1,\tilde{\rho}_1,\epsilon) }\\
&=& \gamma_1^4 \left[ 8 \epsilon (1
+\tilde{\rho}_1)\left(1+(1+\tilde{\rho}_1)(2+6 \tilde{\rho}_1+12
\gamma_1^2 \epsilon)\right) + 12  \tilde{\rho}_1^2 \epsilon +
\frac{16}{3} \gamma_1^4 \tilde{\rho}_1^3 \epsilon^3 \right].
\end{eqnarray*}

Theorem \ref{theo:2bras2}$(b)$  follows by minimizing \tcr{$(\gamma_1,\tilde{\rho}_1,\epsilon) \longmapsto c(\gamma_1,\tilde{\rho}_1,\varepsilon)$ under the}
the constraints:
$$
 \epsilon\leq 1/3, \,
\frac{1}{3\sqrt{2} \tilde{\rho}_1} \leq \gamma_1^2 \leq \frac{3}{2(1
+\tilde{\rho}_1)},	 \, \tilde{\rho}_1 \leq 227/232.
$$
 The ``best" upper bound was obtained by setting
$ 
\gamma_1= 0.89, \tilde{\rho}_1=0.38,\epsilon=1/9,
$
leading to the regret upper bound
$$
\bar{R}_n \leq 44 \sqrt{n}.
$$\end{proof}

\subsection{Proof of Theorems \ref{theo:2bras1} and \ref{theo:2bras2} $(a)$}\label{subsec:prooftheo1}
We prove these  results together. We thus consider $\gamma_1\in(0,1)$, $\rho_1\in(0,1)$ and $\sigma\in[0,1]$. A variant of Proposition \ref{prop:increment} concerning the increase of exponent is still valid. First, it can be observed that if we set $\varepsilon_r=r-1/2$ (so that $\alpha_r=\pi/2$), then Lemma \ref{lemma:sum} can be applied with $\tilde{n}\ge (\frac{\pi}{2}\gamma_1)^{-2}$. Thus, we set $n_0(\lambda):=\lfloor \frac{\lambda^2}{\pi^2}\rfloor+1$ with $\lambda\ge 2\gamma_1^{-1}$. After a simple adaptation of the proof of Proposition \ref{prop:increment}, it can be deduced that for every $r\ge1$, 
 $$
\sup_{n \geq n_0(\lambda)} \mathbb{E} \znr \leq \mathbb{E} \znzr + 
\frac{2r}{\pi} \left[\tilde{\rho}_1 + h_r(\gamma_{n_0(\lambda)})+ \pi \sup_{n \geq n_0(\lambda)} \znrr \right].$$
By an iteration, it follows by using the fact that $\pi\ES[Z_{n_0(\lambda)}^{(i)}]\le \pi\gamma_{n_0(\lambda)}^{-1}\le  \gamma_1^{-1}(\lambda+1)$ that for every $r\ge 1$, some constants $C_r^1(\lambda)$ and $C_r^2(\lambda)$ exist (depending only on $\sigma$, $\gamma_1$ and $\rho_1$) such that, 
\begin{equation}\label{dsuqoiu}
\sup_{n \geq n_0(\lambda)} \pi \mathbb{E}[Y_n] \leq C_r^1(\lambda) +C_r^2(\lambda) \pi \sup_{n \geq n_0(\lambda)}  \ES\znrr.
\end{equation}
It remains to upper bound $\sup_{n \geq n_0(\lambda)}  \ES \znr$ for $r$ large enough. Once again, a simple adaptation of the proof of Lemma \ref{lemme:decompzn3} for $r\ge 3$ yields:
$$ \ES[\Delta Z_{n+1}^{(r)}|{\cal F}_n]=\gamma_{n+1}(1-X_{n})^{r-1} P_{n}^{(r)}(X_{n}) + \Delta R_{n}^{(r)}.
$$
with 
\begin{align}
\lefteqn{
P_n^{(r)}(x)} \nonumber \\ &= \frac{(1-x)^2}{\gamma_{n+1}}(\varepsilon_n-r\pi x)-r \tilde{\rho}_1(1-x)\kappa_{\sigma}(x) 
+\binom{r}{r-2}\left(x(1-x)^2  p_1+x^2(1-x)p_2\right)\nonumber\\& +\gamma_{n+1}\binom{r}{r-3}\left(-x(1-x)^2p_1+x^3p_2\right)
\end{align}
and $\Delta R_n^{(r)}\le C_r\gamma_{n+1}^3$ (where $C_r$ does not depend on $\pi$). We want to prove that $P_n^{(r)}$ is negative on $[0,1-\xi_n]$ with $\xi_n=\xi\gamma_{n+1}\in(0,1)$ where $\xi$ is a constant to be calibrated. We follow the lines of the proof of Proposition \ref{prop:etudePn}, but we can use some rougher arguments since we are not looking  for explicit constants. First, $P_n^{(r)}(0)=\frac{\varepsilon_n}{\gamma_{n+1}}-r\tilde{\rho_1}\kappa_\sigma(0)$, so that:
$$P_n^{(r)}(0)<0\quad \Longleftrightarrow\quad \gamma_1\rho_1>\frac{1}{r\sqrt{2}(1-\sigma p_2)}.$$
On the one hand, for every $\sigma<1$, it is possible to find an $r$ sufficiently large for which this condition holds. On the other hand, when $\sigma=1$ (case of Theorem \ref{theo:2bras1}), we then need to assume that a $\delta>0$  exists such that $p_2<1-\delta$ (in this case, the condition is satisfied if $r>(\gamma_1\rho_1\sqrt{2}\delta)^{-1}$). For such an $r$, it can be observed that the leading coefficient $a_n^{(r)}$ (related to $x^3$) is:
$$a_n^{(r)}=\left(-\frac{r}{\gamma_{n+1}}+\binom{r}{r-2}+r\sigma\tilde{\rho}_1-\gamma_{n+1}\right)\pi.$$
It can therefore be deduced that $a_n^{(r)}$ is negative for every  $n\ge n_1^\sigma$  where:
$$ n_1^\sigma:=\left\lfloor\gamma_1^2\left(\frac{r-1}{2}+\sigma\tilde{\rho}_1\right)^2\right\rfloor.$$
Assume that $\lambda\ge \sqrt{n_1^\sigma}$ in order to obtain $n_0(\lambda)\ge n_1^\sigma$.  Since $P_n^{(r)}(1)=\gamma_{n+1}\binom{r}{r-3}p_2>0$ and ${\rm deg}(P_n^{(r)})=3$, it follows that $P_n^{(r)}$ has exactly one root in $(0,1)$ for every $n\ge {n}_0$ and that  $P_n^{(r)}$ is negative on $[0,1-\xi_n]$ as soon as $P_n^{(r)}(1-\xi_n)<0$. Let $n$ be such that $\xi\gamma_{n+1}\le 1/2$. Then, some rough estimations yield that $P_n^{(r)}(1-\xi_n)$ is negative if
$$ \frac{r\pi}{2}\xi^2-c_r \xi-1>0,$$
where $c_r$ is a constant that does not depend on $\pi$. We then check that another constant $\eta_r$  exists such that the previous property is fulfilled if 
$\xi\ge \eta_r/\pi$. Then, $P_n^{(r)}(1-\frac{\eta_r}{\pi}\gamma_{n+1})<0$ is negative as soon as $\xi\gamma_{n+1}<1/2$. This is true for every $n\ge n_0(\lambda)$ as soon as 
$\lambda\ge 2\gamma_1 \eta_r$. We can conclude  from what preceeds that an $r\ge 3$ and $\lambda>0$ exist such that for every $n\ge n_0(\lambda)$, for every $(p_1,p_2)\in[0,1]^2$, such that $p_1>p_2$ (resp. $p_1>p_2$ and $p_2<1-\delta$) if $\sigma<1$ (resp. if $\sigma=1$) 
$$\ES[\Delta Z_{n+1}^{(r)}]\le \gamma_{n+1}\sup_{t\in[1-\frac{\eta_r}{\pi}\gamma_{n+1},1]}(1-t)P_{n}^{(r)}(t) + C_r\gamma_{n+1}^3.$$
Using $\gamma_{n+1}\le \pi/\lambda$ if $n\ge n_0(\lambda)$,  a constant $C_\lambda$ exists such that on  
$$\forall t \in [1-\frac{\eta_r}{\pi}\gamma_{n+1},1] \qquad P_n^{(r)}(t)\le C_\lambda\gamma_{n+1}/\pi.$$
Under the previous conditions,  we deduce
$$\sup_{\pi}\left(\pi \sup_{n \geq n_0(\lambda)}  \ES\znr\right)\le \sup_{\pi}\left(\pi \sum_{n\ge n_0}C{\gamma_{n+1}^3}\left({\pi}^{-2}+\pi^{-1}\right)\right)<+\infty.$$
The result follows by plugging this inequality into \eqref{dsuqoiu}.
%
%
%
%

\section{Almost sure and weak limit of the over-penalized bandit}\label{sec:appendix_multi}

We provide here the proofs of Propositions \ref{prop:multi1} and \ref{prop:multiweak}.
For the sake of simplicity, we restrict our study to $\sigma=1$ (always over-penalization of the bandit), and the argument can be adapted for any values of $\sigma \in (0,1]$.

\subsection{A.s. convergence of the multi-armed bandit (Proposition \ref{prop:multi1})}

Recall first that $X_{n} = (X_{n,1},...,X_{n,d})$, the multi-armed penalized bandit \eqref{eq:pen_bandit}
makes it possible  to define for $i \in \lbrace 2,...,d\rbrace$,
\begin{eqnarray*}
X_{n+1,i} &=& X_{n,i} + \gamma_{n+1}h_{i}(X_{n}) + \gamma_{n+1}\rho_{n+1}\kappa_{i}(X_{n}) +  \gamma_{n+1}\Delta M_{n+1,i},
\end{eqnarray*} 
where the main part of the drift $h_{i}$ is defined as
$$
h_{i}(x_{1},...,x_{d}) = (1-x_{i})x_{i}p_{i} - x_{i}\sum\limits_{j \neq i}x_{j}p_{j},$$
and the penalty drift is
$$
\kappa_{i}(x_{1},...,x_{d}) = -x_{i}^{2}(1-p_{i}) + \frac{1}{d-1}\sum\limits_{j \neq i}x_{j}^{2}(1-p_{j}).
$$
Hence, the martingale increment is simply obtained as

\begin{eqnarray*}
&\Delta M_{n+1,i}& = ((1-X_{n,i})1_{V_{n+1,i},A_{n+1,i}}-X_{n,i}\sum\limits_{j \neq i}1_{V_{n+1,j},A_{n+1,j}} - h_{i}(X_{n}))\\ &-& \rho_{n+1}(X_{n,i}1_{V_{n+1,i},A_{n+1,i}^{c}} - \frac{1}{d-1}\sum\limits_{j \neq i}X_{n,j}1_{V_{n+1,j},A_{n+1,j}^{c}} +\kappa_{i}(X_{n}))
 \end{eqnarray*}
 
\begin{proof}[Proof of  Proposition \ref{prop:multi1}]

\noindent We start by (i) and identify the stationary points of the ODE method. The ODE $ \dot{ x} = h(x) $ possesses a finite number of equilibria that can be easily identified. 
 We begin by solving the equation $h_{1}(x) =0$. Since $$h_{1}(x) =x_{1}\sum\limits_{i=2}^{d}x_{j}(p_{1}-p_{j})\geq 0,$$
  we either have $x_{1}=1$ and $x_{2}=...=x_{d}=0$ or $x_{1}=0$.\\ Then, the equation $h_{2}(x)=0$ with $x_{1}=0$ may be reduced to 
  $$x_{2}\sum\limits_{i=3}^{d}x_{j}(p_{2}-p_{j})\geq 0.$$ The same argument leads to $x_{2}=1$ or $x_{2}=0$ and a straightforward recursion shows that  the equilibria of the ODE are $(\delta^{i})_{1 \leq i \leq d}$, with $(\delta^{i})_{1 \leq i \leq d}$  defined as
\begin{eqnarray*}
\delta^i_{i} = 1 \quad\text{and}\quad \delta^i_{j}=0\quad \forall j\neq i.
\end{eqnarray*}
  
Let us emphasize that to discriminate among these equilibria, it is not possible to use the second derivative criterion that relies on $\left( \displaystyle \frac{\partial h_{i}}{\partial x_j}\right)_{i,j}$ to establish their stability.
Instead, it is possible to check that $\delta^{1}$ fulfills the Lyapunov certificate with the function $V(x) = (x_2^2+\ldots + x_d^2)$. If we denote $h=(h_1,\ldots, h_d)$, we then have:
$$
\langle \nabla V(x) , h(x) \rangle = \sum_{j=2}^d x_j^2 \sum_{k \neq j} x_k (p_j-p_k).
$$
Considering $x$ in a closed neighborhood of $\delta^{1}$ defined as $x_j \leq \epsilon/d, \, \forall j \geq 2$ (implying that $x_1>1-\epsilon$), we see that:
\begin{eqnarray*}
\langle \nabla V(x) , h(x) \rangle & = & x_1 \sum_{j=2}^d x_j^2 (p_j - p_1) + \sum_{k=2}^d x_k^2 \sum_{j \neq k, j >1}x_j (p_k-p_j) \\
& \leq & -(1-\epsilon) (p_1-p_2)\sum_{j=2}^d x_j^2 + \epsilon \sum_{j=2}^d x_j^2,
\end{eqnarray*}
and the term above is negative as soon as $\epsilon$ is chosen such that:
$$
\epsilon \leq \frac{1}{p_1-p_2+1}.
$$
In contrast, the other equilibria $(\delta^j), j \neq 1$ are unstable: this can be easily deduced from the unstability of the two-armed bandit by testing the first arm vs. the arm $j$.

Since the martingale increment $\Delta M_{n+1,i}$ is uniformly bounded, we can apply the Kushner-Clark theorem (see  \cite{KushnerClarck}) and can conclude that $(X_{n,i})_{n \geq 0}$ either converges to  1 or 0 a.s. As a consequence, it is also true that  $(X_n)_{n \geq 0}$ converges a.s.
We now  make this limit explicit and show that $(X_{n})_{n \geq 0}$ converges toward $(1,...,0)$ a.s.
We start by noticing that $h_{1}(x) =x_1 \sum_{j \geq 2} x_j (p-1-p_j) \geq 0$, which implies that:
\begin{eqnarray}\label{eq:contr}
X_{n,1} &\geq & X_{0,1}+ \sum\limits_{j =1}^{n-1}\gamma_{j}\rho_{j}\kappa_{j-1,1}(X_{j-1}) +  \sum\limits_{j =1}^{n-1}\gamma_{j}\Delta M_{j}.
\end{eqnarray}

\noindent The martingale increment $\Delta M_{j}$ is bounded and a large enough $C$ exists such that  $\Delta M_{j}\leq \sqrt{C}$. This implies that:

$$\left\Vert \sum\limits_{j =1}^{n-1}\gamma_{j}\Delta M_{j}\right\Vert^{2}_{L^2} \leq C\sum\limits_{j =1}^{n-1}\gamma_{j}^{2}\leq C\sup\limits_{j \in \mathbb{N}}\left(\frac{\gamma_{j}}{\rho_{j}}\right)\sum\limits_{j =1}^{n-1}\gamma_{j}\rho_{j}.$$
\noindent Since $\sum_{}\rho_j \gamma_j = + \infty$, we can deduce that 
$$
\lim_{n \longrightarrow + \infty} \frac{\mathbb{E} \left[ \displaystyle\sum_{j=1}^n \gamma_j \Delta M_j\right]^2}{\sum\limits_{j =1}^{n-1}\gamma_{j}\rho_{j}} = 0 \qquad \text{so that} \qquad \limsup\limits_{n\rightarrow\infty}\dfrac{\sum\limits_{j =1}^{n-1}\gamma_{j}\Delta M_{j}}{\sum\limits_{j =1}^{n-1}\gamma_{j}\rho_{j}} \geq 0.
$$

\noindent We now consider an event $\omega \in \lbrace X_{\infty,1}=0\rbrace$. We have:
$$
\lim_{n \longrightarrow + \infty} \kappa_{1}(X_{n}(\omega)) = \frac{1}{d-1}\sum_{k \geq 2} (1-p_k) X_{\infty,k}(\omega)^2,
$$
and 
according to the Toeplitz Lemma we deduce that

$$\lim\limits_{n\rightarrow\infty}\dfrac{\sum\limits_{j =1}^{n-1}\gamma_{j}\rho_{j}\kappa_{1}(X_{j-1}(\omega))}{\sum\limits_{j =1}^{n-1}\gamma_{j}\rho_{j}} = \frac{1}{d-1}\sum\limits_{k \geq 2}(1-p_{k}) X_{\infty,k}(\omega)^{2}>0.$$

\noindent Putting together this last remark with Equation \eqref{eq:contr} leads to the conclusion

$$\limsup\limits_{n\rightarrow\infty}\dfrac{X_{n,1}(\omega)}{\sum\limits_{j =1}^{n-1}\gamma_{j}\rho_{j}} > 0.$$

\noindent We obtain a contradiction with the boundedness of $(X_n)_{n \geq 1}$ and conclude that $\mathbb{P}(X_{\infty,1} = 0) = 0$.
For (ii), we refer to \cite{Lamberton_Pages} since the arguments here are similar.
\end{proof}

\subsection{Weak convergence of the normalized bandit (Proposition \ref{prop:multiweak}) }

The proof of the weak convergence follows the lines of \cite{Lamberton_Pages}. The idea is to prove the tightness of the pseudo-trajectories associated to the normalized sequence and to then show that any weak limit of this sequence is a solution of the martingale problem $({\cal L},{\cal C}^1_K(\ER_+,(\ER_+)^{d-1})$ where ${\cal L}$ is the infinitesimal generator defined in Proposition \ref{prop:multiweak}. Then, proving that uniqueness holds for the solutions of the martingale problem and for the invariant distribution, the convergence follows. Here, we choose to only detail the key step of the characterization of the limit. The rest of the proof can be obtained by a simple generalization of that of  \cite{Lamberton_Pages}.

\begin{proposition} Let $f$ be a continuously differentiable function with compact support in $\mathbb{R}_{+}^{d-1}$. We have 
$$\mathbb{E}(f(Y_{n+1,2},...,Y_{n+1,d})-f(Y_{n,2},...,Y_{n,d})|{\cal F}_n) = \gamma_{n+1}\mathcal{L}_d f(Y_{n,2},...,Y_{n,d})+o_{P}(1),$$
where $\mathcal{L}_d$ is the PDMP generator defined in \eqref{generateur_dbras} and ${\cal F}_n=\sigma(Y_k,k\le n)$.
\end{proposition}

\begin{proof} Since the proof does not depend on $\sigma$, we assume that $\sigma=1$ for the sake of clarity. We first give an alternative expression for the variables $Y_{n,i}$ for $i\geq 2$.
\begin{eqnarray*}
Y_{n+1,i}&=& Y_{n,i} + \gamma_{n+1}\left(\frac{1-p_{1}}{d-1}-(p_{1}-p_{i})Y_{n,i}\right) +  \gamma_{n+1}C_{n,i}-g\Delta M_{n+1,i},
\end{eqnarray*}

\noindent where $C_{n,i} = (\kappa_{i}(X_{n})-\frac{1-p_{1}}{d-1})+Y_{n,i}(p_{1}-p_{i}+(\epsilon_{n}+\frac{\rho_{n}}{\rho_{n+1}}(p_{i}-\sum\limits_{j\neq i}X_{n,j}p_{j}))) = o_{P}(1)$
since $(\epsilon_n)_{n \geq 0}$ converges 0 and $(X_{n,i})_{n \geq 0}$ converges to 0 in probability for $i\geq 2$.
 We rewrite this as follows
\begin{eqnarray*}
Y_{n+1,i}&=& Y_{n,i} + \gamma_{n+1}\left(\frac{1-p_{1}}{d-1}-(p_{1}-p_{i})Y_{n,i}+C_{n,i}\right) +  G_{n,i}+g\Delta \tilde{M} _{n+1,i},
\end{eqnarray*}

where $G_{n,i} = g(1-X_{n,i})( 1_{V_{n+1,i},A_{n+1,i}}-X_{n,i}p_{i})$ and $\Delta \tilde{M} _{n+1,i} = \Delta M_{n+1,i}-G_{n,i}$. \\ We consider a function $f \in \mathcal{C}^1(\mathbb{R}_{+}^{d-1})$ with a compact support.
\begin{eqnarray*}
f(Y_{n+1})-f(Y_{n}) = \sum\limits_{i=2}^{d}f(Y_{n+1,2},...Y_{n+1,i},....Y_{n+1,d})-f(Y_{n,2},...Y_{n,i},....Y_{n,d}).
\end{eqnarray*}

\noindent We will use the following notation $F_{i}(Y_{k}) = f(Y_{n,2},...Y_{k,i},....Y_{n,d})$. This means that the first $i-1$ variables are $(Y_{n,2},Y_{n,3},...)$ and the last $d-i$  ones are: $(Y_{n+1,i+1},Y_{n+1,i+2},...,Y_{n+1,d})$. We have:
\begin{eqnarray*}
F_{i}(Y_{n+1,i})-F_{i}(Y_{n,i}) = F_{i}(Y_{n+1,i})-F_{i}(\overline{Y}_{n,i})+F_{i}(\overline{Y}_{n,i})-F_{i}(Y_{n,i}),
\end{eqnarray*}

\noindent where $$\widetilde{Y}_{n,i} =  Y_{n,i} + \gamma_{n+1}\left(\frac{1-p_{1}}{d-1}-(p_{1}-p_{i})Y_{n,i}+C_{n,i}\right),$$ and $$\overline{Y}_{n,i} = \widetilde{Y}_{n,i}  + G_{n,i}.$$ 

\noindent We begin by writing:
\begin{eqnarray*}
 F_{i}(Y_{n+1,i})-F_{i}(\overline{Y}_{n,i}) = \partial_{i}F_{i}(\widetilde{Y}_{n,i})\Delta \widetilde{M} _{n+1,i} + \gamma_{n+1} V_{n+1,i},
\end{eqnarray*}
\\
\noindent where the first order Taylor approximation formula yields:
$$ \exists \, \theta\in [0,1]: \qquad V_{n+1,i} = \left[F_{i}(\widetilde{Y}_{n,i}+\theta\Delta \widetilde{M} _{n+1,i})- F_{i}(\widetilde{Y}_{n,i})\right]\Delta \widetilde{M} _{n+1,i}.$$
As a consequence, $V_{n+1,i} =o_{P}(1)$  and we are now going to prove that:
\begin{eqnarray*}
\mathbb{P}-\lim\limits_{n\rightarrow\infty}\mathbb{E}\left(\frac{ F_{i}(Y_{n+1})-F_{i}(Y_{n})- \gamma_{n+1}\mathcal{A}_{i}F_{i}(Y_{n})}{\gamma_{n+1}}\vert \mathcal{F}_{n}\right) =0,
\end{eqnarray*} 

\noindent where:
\begin{eqnarray*}
\mathcal{A}_{i}f(Y_{2},...,Y_{d}) &=&\frac{p_{i}Y_{i}}{g}(f(Y_{2},...,Y_{i}+g,...,Y_{d})-f(Y_{2},...,Y_{i},...Y_{d}))\\
&+&\left(\frac{1-p_{1}}{d-1}-p_{1}Y_{i}\right)\partial_{i}f(Y_{2},...,Y_{d}).
\end{eqnarray*}
We compute:
\begin{eqnarray*}
\mathbb{E}(F_{i}(\overline{Y}_{n,i}))\vert \mathcal{F}_{n,i}) &=& p_{i}X_{n,i}F_{i}(\widetilde{Y}_{n,i}+g(1-X_{n,i})(1-p_{i}X_{n,i}))\\ 
&+& (1-gp_{i}X_{n,i})F_{i}(\widetilde{Y}_{n,i}-gp_{i}X_{n,i}(1-X_{n,i}))).
\end{eqnarray*}

\noindent  Let us decompose the r.h.s. of the above equation into two parts,
denoted by: \begin{equation}\label{eq:Fni}F _{n,i} = p_{i}X_{n,i}(F_{i}(\widetilde{Y}_{n,i}+g(1-X_{n,i})(1-p_{i}X_{n,i}))-F_{i}(Y_{n,i})),\end{equation}
 and \begin{equation}\label{eq:Gni}G_{n,i}=  (1-gp_{i}X_{n,i})(F_{i}(\tilde{Y}_{n,i}-gp_{i}X_{n,i}(1-X_{n,i})))-F_{i}(Y_{n,i})).\end{equation}

Note that \eqref{eq:Fni} is the jump part of the PDMP and \eqref{eq:Gni} is the deterministic one. 
If $i \geq 2$,  $(X_{n,i})_{n \geq 1}$ converges to 0 in probability and  $\rho_n{\gamma_{n+1}}^{-1}=g + o(\rho_n)$. Thus:

\begin{eqnarray*}
{\gamma_{n+1}}^{-1} F _{n,i} &=& {\gamma_{n+1}}^{-1}\rho_n p_{i}Y_{n,i}(F_{i}(Y_{n,i}+g+o_P(1))-F_{i}(Y_{n,i}))\\
&=& \frac{p_i Y_{n,i}}{g}\left(1+o(\rho_n)\right)\left[ F_{i}(Y_{n,i}+g+o_P(1))-F_{i}(Y_{n,i})\right].
\end{eqnarray*}
As a consequence,  the asymptotic behavior of \eqref{eq:Fni} is given by
\begin{eqnarray*}
\mathbb{P}-\lim\limits_{n\rightarrow\infty}\left(\frac{F_{n,i}}{\gamma_{n+1}}-p_{i}Y_{n,i}\frac{F_{i}(Y_{n,i}+g)-F_{i}(Y_{n,i})}{g} \right) =0.
\end{eqnarray*}
We now study \eqref{eq:Gni} and compute:
\begin{eqnarray*}
\widetilde{Y}_{n,i}-gX_{n,i}(1-p_{i}X_{n,i}) &=& Y_{n,i} + \gamma_{n+1}\left(\frac{1-p_{1}}{d-1}-p_{1}Y_{n,i}\right)\\
 &+&\gamma_{n+1}p_{i}Y_{n,i}-gp_i X_{n,i}(1-X_{n,i})+\gamma_{n+1}C_{n,i}\\
 &=& Y_{n,i} + \gamma_{n+1}\left(\frac{1-p_{1}}{d-1}-p_{1}Y_{n,i}\right)\\
&&  +\underbrace{\gamma_{n+1}p_{i}Y_{n,i} -g\rho_n p_i Y_{n,i}(1-X_{n,i})+\gamma_{n+1}C_{n,i}}_{:=\gamma_{n+1} \widetilde{C}_{n,i}},
\end{eqnarray*}

\noindent where $g\rho_n = \gamma_n$. Since $\widetilde{C}_{n,i}$ converges to $0$ in probability, we obtain:
\begin{eqnarray*}
\lefteqn{
\gamma_{n+1}^{-1} G_{n,i}}\\
 &=& \gamma_{n+1}^{-1}(1+o(\rho_n))\left( F_{i}(Y_{n,i} + \gamma_{n+1}\left[\frac{1-p_{1}}{d-1}-p_{1}Y_{n,i}\right]+\gamma_{n+1}\tilde{C}_{n,i})-F_{i}(Y_{n,i})\right)\\
&=& \left[\frac{1-p_{1}}{d-1}-p_{1}Y_{n,i}\right] \dfrac{\left(F_{i}(Y_{n,i} + \gamma_{n+1}\left[\frac{1-p_{1}}{d-1}-p_{1}Y_{n,i}\right]+\gamma_{n+1}\tilde{C}_{n,i})-F_{i}(Y_{n,i})\right)}{\gamma_{n+1} \left[\frac{1-p_{1}}{d-1}-p_{1}Y_{n,i}\right]}\\
&+&o_P(1).
\end{eqnarray*}
We finally obtain the limiting behavior of \eqref{eq:Gni}:

$$\mathbb{P}-\lim\limits_{n\rightarrow\infty}\left(\frac{G_{n,i}}{\gamma_{n+1}}-\left(\frac{1-p_{1}}{d-1}-p_{1}Y_{i}\right)\partial_{i}F_{i}(Y_{n,i})\right) =0.$$
This ends the proof of the proposition.
\end{proof}

\section{Ergodicity of the PDMP}\label{sec:ergodicity}

From now on, the variable $(X_t)_{t \geq 0}$ will refer to a trajectory of the PDMP associated with the normalized (over)-penalized bandit and bearing no relation to the multi-armed bandit  sequence $(X_n)_{n \geq 1}$. 

\subsection{Wasserstein results}
We begin the study of the ergodicity of the PDMP whose infinitesimal generator is \eqref{generateur} with some computations of the moments of the process.

\begin{lemma}\label{prop:moments} Let $(X_{t})_{t\geq 0}$ be a Markov process, whose generator $\mathcal{L}$ is defined by (\ref{generateur}). If $\pi :=b-cg>0$, then $\sup\ES[(X_{t}^x)^p]\le C(1+|x|^p)$. In particular, the invariant distribution $\pi$ has moments of any order and
$$\forall t \geq 0 \qquad \mathbb{E}(X_{t}) = \frac{a}{\pi}+\left(\mathbb{E}( X_{0})-\frac{a}{\pi}\right)e^{-t\pi}$$
\end{lemma}

\begin{proof}
 Let us define $f_{p}(x) = x^p$. We have:
\begin{eqnarray}\label{eq:moment}
\mathcal{L}f_{p}(x) &=& p(a-bx)x^{p-1}+cx((x+g)^p-x^p)\nonumber\\
&=& -p\pi f_{p}(x) + paf_{p-1}(x)+c\sum\limits_{k=0}^{p-2}C_{p}^{k}g^{p-k}f_{k+1}(x),
\end{eqnarray}
where we adopt the convention $\Sigma_{\emptyset}=0$.
 If we now define  $\alpha_{p}(t) = \mathbb{E}(X_{t}^{p})$, the previous relation shows that $\alpha_{p}$ satisfies the ODE for any integer $p \geq 1$ defined by
 $$
 \alpha_p(t)' +  p \pi \alpha_p(t)=  p a \alpha_{p-1}(t)+c \sum_{k=0}^{p-2} C_p^k g^{p-k} \alpha_{k+1}(t).
 $$
  For example, with $p=1$ we have $\alpha_{1}'(t) = -\pi\alpha_{1}(t) + a$, which implies that
\begin{eqnarray*}
\alpha_{1}(t) = \frac{a}{\pi}+\left(\mathbb{E}( X_{0})-\frac{a}{\pi}\right)e^{-t\pi}.
\end{eqnarray*}
\noindent The control of the moments of order $p>1$ then follows from a recursion.
\end{proof}

\subsubsection{Rescaled two-armed bandit \& Theorem \ref{theo:pdmp}}

In the following, we will exploit Equation \eqref{eq:moment} to obtain a suitable upper bound of the Wasserstein distance $\mathcal{W}_p$ between the law of $X_t$ and the invariant measure $\mu_{\infty}$ of the PDMP. For this purpose, we note that the generator \eqref{generateur} possesses the stochastic monotonicity property, \textit{i.e.},
 a coupling $(X,Y)$  exists starting from $(x,y)$ (with $x > y$) such that $X_{t} \geq Y_{t}$ for any $t\geq 0$. The increase of the jump rate (with respect to the position) and the positivity of the jumps are of prime importance for this property.
Such a coupling could be built as follows: we only allow simultaneous jumps of both components or a single jump of the highest one (see (\cite{BCG+}) for a similar procedure). The generator of this coupling $(X,Y)$ starting from $(x,y)$ with $x>y$ is given by: 
\begin{eqnarray}\label{generateur_couplage_wass}
\lefteqn{\LW f(x,y) = (a-bx)\partial_{x}f(x,y)+(a-by) \partial_{y}f(x,y)}\nonumber \\
&& + cy\left(f(x+g,y+g)-f(x,y)\right) +	 c(x-y)\left(f(x+g,y)-f(x,y)\right)
\end{eqnarray}
 with a symmetric expression when $y>x$. We now prove the main result.

\begin{proof}[Proof of theorem \ref{theo:pdmp}]
Let $\mu_0$ be  a probability on $\ER_+^*$ and designate $\mu_\infty$ as the invariant distribution of the PDMP. Set 
$${\cal C}_t=\{\nu\in{\cal P}(\ER^2),\nu(dx\times\ER_+)=\mu_t(dx),\nu(\ER_+\times dy)= \mu_{\infty}(dy)\}.$$ 
For any $\nu\in{\cal C}$, let $(X_{t},Y_{t})_{t \geq 0}$ denote the Markov process driven by  (\ref{generateur_couplage_wass}) starting from $\nu$.  From the definition of $\mathcal{W}_p$ and the stationary of  $(Y_t)$, we have for any $t$:
$$\mathcal{W}_{p}(\mu_{t},\mu_{\infty})\leq \inf\{\nu\in{\cal C}_0, \left(\int_{\ER_+^2} \ES[\vert X_{t}^x-Y_{t}^y\vert^p]\nu(dx,dy)\right)^{\frac{1}{p}}.$$
At the price of a potential exchange of the coordinates, we can now work with some deterministic starting points $x$ and $y$ such that $x>y>0$. Owing to the monotonicity of $\LW$, we thus have for any $p\ge1$
$$\mathbb{E}(\vert X_{t}^x-Y_{t}^y\vert)^p = \mathbb{E}(X_{t}^x - Y_{t}^y)^p.$$
Assume now that $p \in \mathbb{N}^*$, we observe that $\LW$ acts on $(x,y) \mapsto (x-y)^p$  as:
$$
\LW (x-y)^p = - p \pi (x-y)^p + p a (x-y)^{p-1} + c \sum_{k=0}^{p-2} C_p^k g^{p-k} (x-y)^{k+1}.
$$
Setting $\beta_p(t) = \mathbb{E} \left|X_{ t}^x - Y_{{t}}^y \right|^p$, we can immediately  check that:
\begin{equation}\label{eq:beta}
\dot{\beta}_p(t) + \pi p \beta_p(t) = \left(p a \beta_{p-1}(t) + c \sum_{k=0}^{p-2} C_p^k g^{p-k} \beta_{k+1}(t)\right).
\end{equation}
When $p=1$, \eqref{eq:beta} implies that:
$
\beta_1(t) =\beta_1(0) e^{-  \pi t}\;\Rightarrow\, \ES[X_t^x-Y_t^y]=(x-y)e^{-\pi t} , 
$
so that:
$$
\mathcal{W}_1(\mu_t,\mu_{\infty}) \leq \mathcal{W}_1(\mu_0,\mu_{\infty}) e^{- t \pi}.
$$
For the lower-bound, we use: 
$$\mathcal{W}_1(\mu_t,\mu_{\infty})\ge \inf\left\{\nu_t \in{\cal C}_t, \left|\int (x-y) \nu_t(dx,dy)\right| \right\}=\left|\ES[X_t^{\mu_0}]-\ES[Y_t^{\mu_\infty}]\right|,$$
which implies that:
$$\mathcal{W}_1(\mu_t,\mu_{\infty})\ge  \left|\int \ES[X_t^{x}-Y_t^{y}] \mu_0(dx)\mu_\infty(dy)\right|=\left|\int (x-y)\mu_0(dx)\mu_\infty(dy)\right|e^{-\pi t}.$$
The lower-bound follows. \smallskip

\noindent Now, let us consider the case $p>1$ (with $p\in\mathbb{N}$). For $p=2$, we have
$$
\left( \beta_2(t) e^{2 \pi t}\right)' e^{-2  \pi t} = (2 a+cg^2) \beta_1(0) e^{- \pi t},
$$
and an integration leads to
$\beta_2(t) e^{2 \pi t} - \beta_2(0) = \frac{2 a+cg^2 }{\pi} \beta_1(0) [e^{ \pi t} - 1].$ As a consequence: 
$$
\beta_2(t) \leq  e^{- 2 \pi t} \beta_2(0)  + \frac{2a+cg^2}{\pi} \beta_1(0) e^{- \pi t}. 
$$
Using the inequalities $\sqrt{u+v}\le \sqrt{u}+\sqrt{v}$ and $\beta_2 \geq \mathcal{W}_2^2$, we thus deduce that:
$$
\mathcal{W}_{2}(\mu_t,\mu_{\infty}) \leq \mathcal{W}_{2}(\mu_0,\mu_{\infty}) e^{-  \pi t} + \sqrt{\frac{2a+cg^2}{\pi}} \sqrt{ \mathcal{W}_{1}(\mu_0,\mu_{\infty})} e^{- \frac{\pi t }{ 2} }.
$$
The result follows when $p=2$ by setting:
$$\gamma_2:=\mathcal{W}_{2}(\mu_0,\mu_{\infty})+ \sqrt{\frac{2a+cg^2}{\pi}} \sqrt{ \mathcal{W}_{1}(\mu_0,\mu_{\infty})}.$$
A recursive argument based on \eqref{eq:beta} shows that a constant $\gamma_p$ exists that only depends on $\mu_0$ and $\mu_{\infty}$ such that: 
\begin{equation*}\label{eq:boundW}{W}_p(\mu_t,\mu_{\infty}) \leq \gamma_p e^{-\frac{\pi}{p} t}. \end{equation*}
%
%

\end{proof}

\subsection{Proof of total variation results }

As mentioned before, the idea is to wait until the paths get close (with a probability controlled by the Wasserstein bound) and then to try to stick them (with high probability). Since  the jump size is deterministic, sticking the paths implies a non trivial coupling of the jump times which is described in the lemma below.

 We begin by establishing the next useful lemma.

\begin{lemma}\label{lemma:stick} Let $\varepsilon>0$ and $t\geq \frac{1}{b}\ln(1+\varepsilon)$. A coupling $(X_{t},Y_{t})_{t\geq 0}$ of paths driven by (\ref{generateur}) exists such that on $A_{x_{0},\varepsilon}$:
\begin{eqnarray*}
\mathbb{P}(X_{t}=Y_{t}, t\geq s) \geq \left(1-\frac{c}{b} x_0\varepsilon-e^{-\frac{a}{b} c s} - \frac{c\varepsilon}{b}\right) \max(0,1-\frac{c}{b}\varepsilon(x_0+g)),
\end{eqnarray*}
\noindent where $A_{x_{0},\varepsilon}= \left\{ (x,y)\vert \frac{a}{b}<x\leq x_{0},  0 < x-y \leq \varepsilon \right\}$.
\end{lemma}
\subparagraph{Proof} Let $\varepsilon>0$ and $(x,y)\in A_{x_0,\varepsilon}$  (in particular, $x>y$). Designate $T_1^x$ and  $T_1^y$  as the first jumps of $(X_t^x)$ and $(X_t^y)$, respectively, and $T_2^x$ as  the second jump of $(X_t^x)$. It can be noted  that: 
$$\mathbb{P}(X_{t}=Y_{t}, t\geq s) \geq \PE(X_{T_1^y}^x=X_{T_1^y}^y, T_1^y\le s).$$
We aim to build a coupling that leads to a sharp lower-bound of the r.h.s. For this purpose, note that if $T_1^x<T_1^y<T_2^x$, the triple $(T_1^x,T_1^y,T_2^x)$ satisfies:
\begin{eqnarray*}
X_{T_1^y}^y=X_{T_1^y}^x\,\Longleftrightarrow\, \frac{a}{b} + \left(y-\frac{a}{b}\right)e^{-b T_{1}^{y}} +g = \frac{a}{b} + \left(X_{T_{1}^{x}}^x-\frac{a}{b}\right)e^{-b(T_{1}^{y}-T_{1}^{x})}. 
\end{eqnarray*}
Considering that $X_{T_1^x}^x=\frac{a}{b}+(x-\frac{a}{b})e^{-{T_1^x}}+g$ and defining $\psi(t)=\frac{1}{b}\ln \left( e^{b t}+\frac{x-y}{g}\right)$, we can verify that  $X_{T_1^y}^y=X_{T_1^y}^x\le s$ and $T_1^x<T_1^y<T_2^x$  as soon as 
\begin{eqnarray*}
T_{1}^{y} =\psi(T_1^x)\le s \quad\text{and}\quad T_2^x\ge \psi(T_1^x),
\end{eqnarray*}
%
since  $\psi(t) \ge t$. 
We are naturally encouraged to consider  $S_1^{x,s}=\psi(T_1^x)1_{\{\psi(T_1^x)\le s\}}$ and it is well known that the law of $(T_1^x,T_1^y)$ can be described through the maximal coupling:
$$T_1^y=\Theta U+(1-\Theta) V_y,\quad \psi(T_1^x)=\Theta U+(1-\Theta) V_x,$$ 
where $V_x, V_y,\Theta$ and $U$ are independent, $U\sim \frac{\PE_{T_1^y}\wedge\PE_{\psi(T_1^x)}}{\|\PE_{S_1^{x,s}}\wedge \PE_{T_1^y}\|_{_{TV}}}$ and $\Theta\sim{\cal B}(p)$ where $p=\|\PE_{S_1^{x,s}}\wedge \PE_{T_1^y}\|_{_{TV}}$. With this coupling, if $q(t,z)=\PE(T_1^z\ge \psi(t)-t)$, the Strong Markov property yields
$$\PE(T_2^x-T_1^x|(T_1^x,T_1^y))=\PE(T_2^x\ge \psi(T_1^x)|T_1^x)=q(T_1^x,X_{T_1}^x) .$$
Since $z\mapsto q(t,z)$ is increasing and $x>a/b$ (from the assumption on $A_{x_0,\epsilon}$), we deduce that $X_{T_1}^x\le x+g$ and it therefore follows that: 
$$\PE(T_2^x>T_1^y|(T_1^x,T_1^y))\ge q(t,x+g)\ge q(0,x+g).$$
given that $t\mapsto \psi(t)-t$ is a non-decreasing function. As a consequence, we obtain that with this coupling:
\begin{equation}\label{eq:ref11}
\PE(X_{T_1^y}^x=X_{T_1^y}^y, T_1^y\le s)\ge q(0,x+g) \PE(\Theta=1)=q(0,x+g) \|\PE_{S_1^{x,s}}\wedge \PE_{T_1^y}\|_{_{TV}}.
\end{equation}
It remains to find a lower bound of the total variation distance involved in the r.h.s. of the above inequality .
Recall that
$$\|\PE_{S_1^{x,s}}\wedge \PE_{T_1^y}\|_{_{TV}}=\int_0^{+\infty} f_y(t)\wedge g_{x,s}(t) dt,$$
where $f_y$ and $g_{x,s}$ denote the densities of $T_1^y$ and $S_1^{x,s}$, respectively. We therefore have:
\begin{eqnarray*}
\forall t>0,\quad f_{y}(t) = c\phi(y,t)e^{-\int\limits^{t}_{0}c\phi(y,u)du}\quad\textnormal{with}\quad \phi(y,t)=\frac{a}{b}+(y-\frac{a}{b})e^{-bt},
\end{eqnarray*}
and  a change of variable yields:
\begin{equation}\label{diepoipr}
\forall t>0,\quad g_x(t)= f_x (\psi^{-1}(t))(\psi^{-1})'(t)1_{\{\psi(0)\le t\le s\}}.
\end{equation}
%
On the one hand, since $(x,y) \in A_{x_0,\epsilon}$, we can check that:
$$\forall t\,\ge0, \quad \phi(x,t)-\varepsilon e^{-bt}\le \phi(y,t)\le \phi(x,t), $$
and we can then conclude that:
$$ \forall t>0,\quad  f_y (t)\ge f_x(t)-\varepsilon e^{-bt}.$$
One the other hand, note that:
$$\forall t> \psi(0),\quad \psi^{-1}(t)=\frac{1}{b}\ln \left(e^{bt}-\frac{x-y}{g}\right)\le t \quad \text{and} \quad
 (\psi^{-1})'(t)= \frac{e^{bt}}{e^{bt}-\frac{x-y}{g}}\ge 1, $$
and we can deduce from \eqref{diepoipr}  that $\forall t\in[ \psi(0),s]$:
$$g_x(t)\ge  c\phi(x,\psi^{-1}(t)) e^{-\int_0^t c \phi(x,s) ds}\ge c\phi(x,t) e^{-\int_0^t c \phi(x,s) ds}=f_x(t).$$
Note that we used that $t\mapsto \phi(x,t)$ is decreasing since $x>a/b$. Thus, 
$$\left(\PE_{T_1^y}\wedge \PE_{S_1^x}\right)(dt)\ge h(t) dt\quad \textnormal{with}\quad h(t)=  (f_x(t)-\varepsilon e^{-bt})1_{\psi(0)\le t\le s} dt.$$
As a consequence, 
$$\|\PE_{S_1^{x,s}}\wedge \PE_{T_1^y}\|_{_{TV}}\ge e^{-\int_0^{\psi(0)} c\phi(x,u) du}-e^{-\int_0^{s} c\phi(x,u) du}-\frac{\varepsilon}{b}.$$
Checking that $\psi(0)\le \varepsilon/b$ and that $\forall t\ge0$, $a/b\le\phi(x,t)\le x\le x_0$, we deduce that 
\begin{align*}
\|\PE_{S_1^{x,s}}\wedge \PE_{T_1^y}\|_{_{TV}}&\ge e^{-\frac{c x_0\varepsilon}{b}}-e^{-\frac{a}{b} c s}-\frac{\varepsilon}{b}\ge 1-\frac{c x_0\varepsilon}{b}-e^{-\frac{a}{b} c s}-\frac{\varepsilon}{b},
\end{align*}
where we used $e^{-u}\ge 1-u$ for $u\ge0$ in the second line. To conclude the proof, it remains to plug this inequality into \eqref{eq:ref11} and to observe that:
$$q(0,x+g)\ge q(0,x_0+g)=e^{-\int_0^{\psi(0)}c\phi(x_0+g,s)ds}\ge 1-c\psi(0)(x_0+g)\ge 1-\frac{c}{b}\varepsilon(x_0+g).$$
\hfill$\square$

We now provide  the proof of the ergodicity w.r.t. the total variation distance.

\begin{proof}[Proof of Theorem \ref{distance var totale}]For any starting distribution $\mu_0$,
\begin{equation}\label{eq:point1}
\|\mu_0 P_t-\mu_\infty\|_{TV}\le \int \|\delta_x P_t-\delta_y P_t\|_{TV}\mu_0(dy)\mu_\infty(dx).
\end{equation}
The idea is to use the Wasserstein coupling during a time $t_1$ and  to then try to stick the paths on the interval $[t_1,t]$ using Lemma \ref{lemma:stick}.   Consider $A_{x_0,\varepsilon}$ defined in Lemma \ref{lemma:stick} and the alternative set $A_{x_0,\varepsilon}^*=\{(x,y),
a/b<y<x_0, 0< y-x\le \varepsilon\}$. Set $B_{x_0,\varepsilon}=A_{x_0,\varepsilon}\cup A_{x_0,\varepsilon}^*$, we have:
\begin{equation}\label{eq:point2}
1-\|\delta_x P_t-\delta_y P_t\|_{_{TV}}\ge \PE(X_t^x=Y_t^y|(X_{t_1}^x,Y_{t_1}^y)\in B_{x_0,\varepsilon})\PE((X_{t_1}^x,Y_{t_1}^y)\in B_{x_0,\varepsilon}).
\end{equation}
Since the Wasserstein coupling preserves the order and since $x>a/b$ $\mu_\infty(dx)$-$a.s.$, it can be noted that $\mu_\infty(dx)$-$a.s.$,
\begin{equation*}
(X_{t_1}^x,Y_{t_1}^y)\in B_{x_0,\varepsilon}\Longleftrightarrow \begin{cases} X_{t_1}^x-X_{t_1}^y\le \varepsilon \textnormal{ and }X_{t_1}^x\le x_0 &\textnormal{if $x\ge y$}\\
 X_{t_1}^y-X_{t_1}^x\le \varepsilon \textnormal{ and }X_{t_1}^y\le x_0&\textnormal{if $x< y$}.
\end{cases}
\end{equation*}
It follows that for every $p>0$, $\mu_\infty(dx)$ almost surely:
\begin{align*}
 \PE((X_{t_1}^x,Y_{t_1}^y)\in B_{x_0,\varepsilon}^c)&\le \PE(|X_{t_1}^x-X_{t_1}^y|> \varepsilon)+\PE(X_{t_1}^x>x_0)+\PE(X_{t_1}^y>x_0)\\
&\le \frac{1}{\varepsilon}\ES[|X_{t_1}^x-X_{t_1}^y|]+\frac{1}{x_0^p}\left(\ES[(X_{t_1}^x)^p]+\ES[(X_{t_1}^y)^p]\right).
\end{align*}
On the basis of Theorem \ref{theo:pdmp} and Lemma \ref{prop:moments}, a constant $C_p$  exists such that $C_p$ depends on $p$, $\mu_0$ and $\mu_\infty$ but not on $t_1$ and satisfies:

$$\int\PE((X_{t_1}^x,Y_{t_1}^y)\in B_{x_0,\varepsilon}^c)\mu_0(dy)\mu_\infty(dx)\le \frac{{\cal W}_1(\mu_0,\mu_\infty)}{\varepsilon}e^{-\pi t_1}+ \frac{C_p}{x_0^p}.$$
Finally,  Lemma \ref{lemma:stick} leads to:
\begin{eqnarray*}
\lefteqn{
\PE(X_t^x=Y_t^y|(X_{t_1}^x,Y_{t_1}^y)\in B_{x_0,\varepsilon})} \\ 
& \ge& \left(1-\frac{c}{b} x_0\varepsilon-e^{-\frac{a}{b} c (t-t_1)} - \frac{c\varepsilon}{b}\right) \left\{ 0 \vee 1-\frac{c}{b}\varepsilon(x_0+g) \right\}
\end{eqnarray*}
so that by plugging the previous inequalities into \eqref{eq:point2} and \eqref{eq:point1}, it can be deduced that for every $p>1$, a constant $\tilde{C}_p$ exists  such that
for every $t\ge 0$, for every $x_0$ and $\varepsilon$ such that $x_0\varepsilon\le b/2c$ (with $x_0>1$ and $\varepsilon\in(0,1))$,
$$\|\mu_0 P_t-\mu_\infty\|_{TV}\le{\tilde C}_p\left(x_0\varepsilon+e^{-\frac{a}{b} c (t-t_1)} +\varepsilon+\frac{1}{\varepsilon}e^{-\pi t_1}+ \frac{1}{x_0^p}\right).$$
If we try to optimize the above bound, we set $t_1= \delta t$, $x_0= C_1 e^{\alpha t}$, $\varepsilon=C_2 e^{-\beta t}$ with $\delta\in(0,1)$ and $\beta>\alpha>0$ and  deduce that a constant $\check C_p$ exists such that:
$$\|\mu_0 P_t-\mu_\infty\|_{TV}\le \check{C}_p \exp\left(-t \left\{ \beta-\alpha \wedge \frac{ca}{b}(1-\delta) \wedge \delta\pi-\beta \wedge \alpha p \right\} \right).$$
We can choose $p$ as large as we want ($\mu_0$ has moments of any order) and thus $\alpha$ arbitrarily small. The result then follows using an optimization on $(\beta,\delta)$.
\end{proof}

\appendix
\section{Technical result for the pseudo-regret upper bound}

\begin{lemma}\label{lemma:sum} Let $\alpha>0$,  $\gamma_1\in(0,1)$ and $\tilde{n}\in\mathbb{N}$ such that $\alpha\gamma_{\tilde{n}}<1$ and $\tilde{n}\ge 1/(\alpha\gamma_1)^2)$. We have:
$$
\forall n \geq \tilde{n} \qquad 
\sum\limits_{j=\tilde{n}}^{n-1}\gamma_{j}\prod\limits_{l=j}^{n-1}(1-\alpha\gamma_{l})\leq \frac{1}{\alpha}$$
\end{lemma}
\subparagraph{Proof} Let $j\ge \tilde{n}$. On the basis of the inequality $\ln(1+x)\geq x$ for $x > -1$, we have

$$\prod\limits_{l=j}^{n-1}(1-\alpha \gamma_{l})= \exp\left(\ln\sum\limits_{l=j}^{n-1}(1-\alpha \gamma_{l})\right)\leq \exp\left(-\sum\limits_{l=j}^{n-1}\alpha\gamma_{l}\right)$$

\noindent Using that $x\mapsto 1/\sqrt{x}$ is decreasing, 
$$\sum\limits_{l=j}^{n-1}\gamma_{l} =\sum\limits_{l=j}^{n-1}\frac{\gamma_1}{\sqrt{l}}\geq \gamma_1\sum\limits_{l=j}^{n-1}\int\limits_{l}^{l+1}\frac{1}{\sqrt{x}}dx = \gamma_1\int\limits_{j}^{n}\frac{1}{\sqrt{x}}dx = 2\gamma_1(\sqrt{n}-\sqrt{j})$$

\noindent so that:

\begin{eqnarray*}
\sum\limits_{j=n_0}^{n-1}\gamma_{j}\prod\limits_{l=j}^{n-1}(1-\alpha \gamma_{l}) &\leq & \gamma_1 e^{-2\alpha\gamma_1\sqrt{n}}\sum\limits_{j=n_0}^{n-1}\frac{e^{2\alpha\gamma_1\sqrt{j}}}{\sqrt{j}}.
\end{eqnarray*}

\noindent Checking that $x\mapsto \frac{1}{\sqrt{x}}e^{\alpha\gamma_1\sqrt{x}}$ is non-decreasing on [$\frac{1}{(\alpha\gamma_1)^2},\infty$) it can be deduced that for any $j\geq n_0,$ 
$$\sum_{j=n_0}^{n-1}\frac{1}{\sqrt{j}}e^{2\alpha\gamma_1\sqrt{j}}\leq \int\limits_{n_0}^{n} \frac{1}{\sqrt{x}}e^{2\alpha\gamma_1\sqrt{x}}dx\le \frac{1}{\alpha\gamma_1}.$$
The lemma follows.

\hfill \cqfd
\bibliographystyle{plain}
\bibliography{GPS_2014}

\end{document}